\documentclass[11pt,,a4paper]{article}
\usepackage{amsmath}\usepackage{epsf,amsfonts,amsthm}\usepackage{amscd,amssymb}
\usepackage{xcolor,epic,eepic}\usepackage{epsfig}
\usepackage{fontenc,indentfirst, delarray,amsfonts,amsmath,amssymb}
\usepackage{rotating}
\usepackage{mathdots}
\usepackage[T1]{fontenc}
\usepackage[matrix,arrow,curve]{xy}
\usepackage{amsmath}
\usepackage{amssymb}
\usepackage{amsthm}
\usepackage{amscd}
\usepackage{amsfonts}
\usepackage{graphicx}
\usepackage{fancyhdr}
\usepackage{dsfont,texdraw}
\usepackage{amsmath}\usepackage{epsf,amsfonts,amsthm}\usepackage{amscd,amssymb}
\usepackage{xcolor,epic,eepic}\usepackage{epsfig}
\usepackage{fontenc,indentfirst, delarray,amsfonts,amsmath,amssymb}
\usepackage{rotating}
\usepackage{amscd}
\usepackage{amsmath}
\usepackage{amsthm}
\usepackage{mathrsfs}
\usepackage{latexsym}
\usepackage{amssymb}
\usepackage{amsfonts}
\theoremstyle{plain}

\usepackage{tikz}
\usepackage{tikz-cd}
\usetikzlibrary{arrows,chains,matrix,positioning,scopes,snakes}
\makeatletter
\tikzset{join/.code=\tikzset{after node path={%
\ifx\tikzchainprevious\pgfutil@empty\else(\tikzchainprevious)%
edge[every join]#1(\tikzchaincurrent)\fi}}}
\makeatother

\tikzset{>=stealth',every on chain/.append style={join},
         every join/.style={->}}
\tikzset{
    >=stealth',
    punkt/.style={
           rectangle,
           rounded corners,
           draw=black, very thick,
           text width=6.5em,
           minimum height=2em,
           text centered},
    pil/.style={
           ->,
           thick,
           shorten <=2pt,
           shorten >=2pt,}
}

\usepackage[T1]{fontenc}
\newcommand{\bee}{\begin{enumerate}}
\newcommand{\eee}{\end{enumerate}}
\newcommand{\benn}{\begin{equation*}}
\newcommand{\eenn}{\end{equation*}}
\newcommand{\be}{\begin{equation}}
\newcommand{\ee}{\end{equation}}
\newcommand{\bean}{\begin{eqnarray}}
\newcommand{\eean}{\end{eqnarray}}
\newcommand{\bea}{\begin{eqnarray*}}
\newcommand{\eea}{\end{eqnarray*}}

\newcommand{\p}{\partial}

\newcommand{\ra}{\rangle}

\newcommand{\N}{\mathbb{N}}

\newcommand{\R}{\mathbb{R}}

\newcommand{\op}[1]{\!\!\mathop{\rm ~#1}\nolimits}

\newcommand{\mbi}{\mathbb{I}}

\mathchardef\za="710B  
\mathchardef\zb="710C  
\mathchardef\zg="710D  
\mathchardef\zd="710E  
\mathchardef\zve="710F 
\mathchardef\zz="7110  
\mathchardef\zh="7111  
\mathchardef\zy="7112 

\mathchardef\zi="7113  
\mathchardef\zk="7114  
\mathchardef\zl="7115  
\mathchardef\zm="7116  
\mathchardef\zn="7117  
\mathchardef\zx="7118  
\mathchardef\zp="7119  
\mathchardef\zr="711A  
\mathchardef\zs="711B  
\mathchardef\zt="711C  
\mathchardef\zu="711D  
\mathchardef\zf="711E 
\mathchardef\zq="711F  
\mathchardef\zc="7120  
\mathchardef\zw="7121  
\mathchardef\ze="7122  
\mathchardef\zvy="7123  
\mathchardef\zvw="7124  
\mathchardef\zvr="7125 
\mathchardef\zvs="7126 
\mathchardef\zvf="7127  
\mathchardef\zG="7000  
\mathchardef\zD="7001  
\mathchardef\zY="7002  
\mathchardef\zL="7003  
\mathchardef\zX="7004  
\mathchardef\zP="7005  
\mathchardef\zS="7006  
\mathchardef\zU="7007  
\mathchardef\zF="7008  
\mathchardef\zW="700A  

\newcommand{\cyclic}{\mathop{\kern0.9ex{{+}
\kern-2.15ex\raise-.25ex\hbox{\Large\hbox{$\circlearrowright$}}}}\limits}

\newcommand{\cE}{{\cal E}}
 \newcommand{\cS}{{\cal S}}

 \newcommand{\cD}{{\cal D}}
 \newcommand{\cO}{{\cal O}}
 \newcommand{\cB}{{\cal B}}

\newtheorem{rem}{Remark}
\newtheorem{theo}{Theorem}
\newtheorem{prop}{Proposition}
\newtheorem{lem}{Lemma}
\newtheorem{cor}{Corollary}
\newtheorem{ex}{Example}

\newtheorem{defi}{Definition}

\newcommand{\h}{\op{Hom}}
\newcommand{\0}{\otimes}

\newcommand{\id}{\op{id}}
\newcommand{\coker}{\op{coker}}
\newcommand{\im}{\op{im}}

\DeclareMathAlphabet{\mathpzc}{OT1}{pzc}{m}{it}

 \newcommand{\cJ}{\mathcal{J}}

 \newcommand{\colim}{\op{colim}}

\begin{document}
\title{\bf Model categorical Koszul-Tate resolution\\ for\\ algebras over differential operators}
\date{}
\author{Gennaro di Brino, Damjan Pi\v{s}talo, and Norbert Poncin\footnote{University of Luxembourg, Mathematics Research Unit, 1359 Luxembourg City, Luxembourg, gennaro.dibrino@gmail.com, damjan.pistalo@uni.lu, norbert.poncin@uni.lu}}

\maketitle

\begin{abstract} Derived $\cD$-Geometry is considered as a convenient language for a coordinate-free investigation of nonlinear partial differential equations (up to symmetries). One of the first issues one meets in the functor of points approach to derived $\cD$-Geometry, is the question of a model structure on the category $\tt C$ of differential non-negatively graded quasi-coherent commutative algebras over the sheaf $\cD$ of differential operators of an appropriate underlying variety. In \cite{BPP1}, we described a cofibrantly generated model structure on $\tt C$ via the definition of its weak equivalences and its fibrations. In the present article -- the second of a series of works on the Batalin-Vilkovisky-formalism -- we characterize the class of cofibrations, give explicit functorial cofibration-fibration factorizations, as well as explicit functorial fibrant and cofibrant replacement functors. We then use the latter to build a model categorical Koszul-Tate resolution for $\cD$-algebraic `on-shell function' algebras.
\end{abstract}

\vspace{2mm} \noindent {\bf MSC 2010}: 18G55, 16E45, 35A27, 32C38, 16S32, 18G10

\thispagestyle{empty}

\section{Introduction}

The study of systems of nonlinear {\small PDE}-s and their symmetries, via the functor of points approach to spaces and varieties, leads to derived $\cD$-stacks, i.e., roughly, locally representable sheaves ${\tt DG_+qcCAlg}({\cal D}_X)\to {\tt SSet}$ valued in the category $\tt SSet$ of simplicial sets and defined on the category ${\tt DG_+ qcCAlg}({\cal D}_X)$ of differential non-negatively graded commutative algebras -- over the sheaf $\cD_X$ of differential operators of a smooth affine scheme $X$ --$\,$, whose terms are quasi-coherent as modules over the function sheaf $\cO_X$ of $X$. The sheaf condition appears a the fibrant object condition of a model structure on the category of the corresponding presheaves. This structure depends on the model structure of the source category, which is equivalent to the category $\tt DG\cD A$ of differential non-negatively graded commutative algebras over the total sections $\cD:=\cD_X(X)=\zG(X,\cD_X)$ of $\cD_X$. In \cite{BPP1}, we defined and studied a finitely generated model structure on $\tt DG\cD A$. In the present paper, we complete its description: we characterize cofibrations as the retracts of the relative Sullivan $\cD$-algebras. Further, we give explicit functorial `{$\,$\small TrivCof -- Fib}' and `{$\,$\small Cof -- TrivFib}' factorizations (as well as the corresponding functorial fibrant and cofibrant replacement functors). The latter are specific to the considered setting and are of course different from those provided, for arbitrary cofibrantly generated model categories, by the small object argument. Eventually, we review the $\cD$-geometric counterpart $\cal R$ of an algebra of on-shell functions and apply our machinery to find a model categorical Koszul-Tate ({\small KT}) resolution of $\cal R$. This resolution is a cofibrant replacement of $\cal R$ in an appropriate coslice category of $\tt DG\cD A$. In contrast with\medskip

\noindent - the classical {\small KT} resolution constructed in coordinates \cite{Bar}, for any regular on-shell irreducible gauge theory (as the Tate extension of the local Koszul resolution of a regular surface), and\smallskip

\noindent - the compatibility complex {\small KT} resolution built in coordinates \cite{Ver}, under regularity and off-shell reducibility conditions (existence of a finite formally exact compatibility complex),\medskip

\noindent the mentioned $\cD$-geometric {\small KT} resolution, obtained from the cofibrant replacement functor of $\tt DG\cD A$, is functorial and exists without the preceding restrictive hypotheses.\medskip

In this series of papers, our final goal is to combine and generalize aspects of Vinogradov's secondary calculus \cite{Vin}, of the homotopical algebraic geometry ({\small HAG}) developed by To\"en and Vezzosi \cite{TV04, TV08}, and the $\cD$-geometry used by Beilinson and Drinfeld \cite{BD}. For Vinogradov, the fundamental category is roughly the homotopy category of the (coslice category under a fixed diffiety or $\cD$-scheme [in particular, under a fixed affine $\cD$-scheme or $\cD$-algebra] of the) category $\tt DG\cD M$ of differential graded $\cD$-modules. In the present paper, we study the homotopy theory of `diffieties' by describing a model structure on $\tt DG\cD A$: we investigate the $\cD$-analog of Rational Homotopy Theory. On the other hand, {\small HAG} deals with the category $\tt DGCA$ of differential graded commutative algebras over a commutative ring. To study partial differential equations, we have to switch to the category of differential graded commutative algebras over the sheaf of noncommutative rings of differential operators of a scheme or variety. Eventually, in comparison with the frame considered by Beilinson and Drinfeld, we aim at dealing not only with $\cD$-schemes, but also with (derived) $\cD$-stacks. We expect this context to be the correct setting for a coordinate-free gauge reduction -- see \cite{PP} and \cite{BPP3} for first results. \medskip

Let us emphasize that the special behavior of the noncommutative ring $\cD$ turns out to be a source of possibilities, as well as of problems. For instance, a differential graded commutative algebra ({\small DGCA}) $A$ over a field or a commutative ring $k$ is a differential graded $k$-module, endowed with a degree zero associative graded-commutative unital $k$-bilinear multiplication, for which the differential is a graded derivation. The extension of this concept to noncommutative rings $R$ is not really considered in the literature. Indeed, the former definition of a {\small DGCA} over $k$ is equivalent to saying that $A$ is a commutative monoid in the category of differential graded $k$-modules. However, for noncommutative rings $R$, the category of differential graded (left) $R$-modules is not symmetric monoidal and the notion of commutative monoid is meaningless. In the case $R=\cD$, we get differential graded (left) $\cD$-modules and these {\it are} symmetric monoidal. But a commutative monoid is not exactly the noncommutative analog of a {\small DGCA} in the preceding sense: the multiplication is only $\cO$-bilinear and, in addition, vector fields act on products as derivations. Further, although we largely avoid sheaves via the confinement to affine schemes -- a necessary restriction, without which no projective model structure would exist on the relevant categories \cite[Ex. III.6.2]{Har} --$\,$, sheaves and quasi-coherence do require a careful approach. Examples of more challenging aspects are the questions of flatness and projectivity of $\cD=\cD_X(X)$ viewed as $\cO=\cO_X(X)$-module, the combination of `finite' and `transfinite' definitions and results, the functorial `{\small TrivCof -- Fib}' and `{\small Cof -- TrivFib}' factorizations...\medskip

Eventually, we hope that the present text and the one of \cite{BPP1} will be considered as self-contained, not only by researchers from different fields, like e.g., homotopical algebra, geometry, mathematical physics, but also by graduate students.\medskip

The paper is organized as follows:
\tableofcontents

\section{Preliminaries}

In the following, we freely use notation, definitions, and the results of \cite{BPP1}. For the convenience of the reader, we nevertheless recall some concepts and propositions in the present section. For explanations on $\cD$-modules, sheaves versus global sections, model categories, small objects, cofibrant generation, as well as on relative Sullivan algebras, we refer the reader to \cite[Appendix]{BPP1}.\medskip

\begin{theo}\label{FinGenModDGDM}
For any unital ring $R$, the category ${\tt Ch}_+(R)$ of non-negatively graded chain complexes of left $R$-modules is a finitely $(\,$and thus a cofibrantly$\,)$ generated model category $(\,$in the sense of \cite{GS} and in the sense of \cite{Hov}$\;)$, with $$I=\{i_k:S_\bullet^{k-1}\to D_\bullet^k,\;k\ge 0\}$$ as its generating set of cofibrations and $$J=\{\zeta_k:0\to D_\bullet^k,\;k\ge 1\}$$ as its generating set of trivial cofibrations. Here $D^k_\bullet$ is the $k$-disc chain complex \be\label{Disc}D^k_{\bullet}: \cdots \longrightarrow 0\longrightarrow 0\longrightarrow \stackrel{(k)}{R} \stackrel{\id}{\longrightarrow} \stackrel{(k-1)}{R}\longrightarrow 0\longrightarrow \cdots\longrightarrow \stackrel{(0)}{0}\;,\ee $S^k_\bullet$ is the $k$-sphere chain complex \be\label{Sphere}S^k_\bullet: \cdots \longrightarrow 0\longrightarrow 0\longrightarrow \stackrel{(k)}{R}\longrightarrow 0\longrightarrow \cdots\longrightarrow \stackrel{(0)}{0}\;,\ee and $i_k,\zeta_k$ are the canonical chain maps. The weak equivalences of this model structure are the chain maps that induce an isomorphism in homology, the cofibrations are the injective chain maps with degree-wise projective cokernel $(\,$projective object in ${\tt Mod}(R)$$\,)$, and the fibrations are the chain maps that are surjective in $(\,$strictly$\,)$ positive degrees. Further, the trivial cofibrations are the injective chain maps $i$ whose cokernel $\coker(i)$ is strongly projective as a chain complex $(\,$strongly projective object $\coker(i)$ in ${\tt Ch}_+(R)$, in the sense that, for any chain map $c:\coker(i)\to C$ and any chain map $p:D\to C$, there is a chain map $\ell:\coker(i)\to D$ such that $p\circ\ell=i$, if $p$ is surjective in $(\,$strictly$\,)$ positive degrees$\,)$.\end{theo}

\begin{prop}\label{EquivShMod} If $X$ is a smooth affine algebraic variety, its global section functor yields an equivalence of symmetric monoidal categories \be\label{ShVsSectqcModDMon}\zG(X,\bullet):({\tt DG_+qcMod}(\cD_X),\0_{\cO_X},\cO_X)\to ({\tt DG\cD M},\0_{\cO},\cO)\;\ee between the category of differential non-negatively graded modules over the sheaf $\cD_X$ of differential operators on $X$, which are quasi-coherent as modules over the function sheaf $\cO_X$, and the category of differential non-negatively graded modules over the ring $\cD=\cD_X(X)$ of global sections of $\cD_X$. The tensor product is taken over the sheaf $\cO_X$ and over the algebra $\cO=\cO_X(X)$, respectively. \end{prop}

\begin{prop}\label{EquivShAlg} If $X$ is a smooth affine algebraic variety, its global section functor induces an equivalence of categories \be\label{ShVsSectqcDAlg} \zG(X,\bullet): {\tt DG_+qcCAlg}(\cD_X)\rightarrow{\tt DG\cD A}\;\ee between the category of differential non-negatively graded $\cO_X$-quasi-coherent commutative algebras over $\cD_X$ and the category of differential non-negatively graded commutative algebras over $\cD$. \end{prop}

\begin{prop} The graded symmetric tensor algebra functor $\cS$ and the forgetful functor $\op{For}$ provide an adjoint pair $${\cal S}:{\tt DG\cD M}\rightleftarrows{\tt DG\cD A}:\op{For}\;$$ between the category $\tt DG\cD M$ and the category $\tt DG\cD A$.\end{prop}

\begin{theo}\label{FinGenModDGDA} The category $\tt DG\mathcal{D}A$ of differential non-negatively graded commutative $\cD$-algebras is a finitely $(\,$and thus a cofibrantly$\,)$ generated model category $(\,$in the sense of \cite{GS} and in the sense of \cite{Hov}$\;)$, with $\cS (I)=\{\cS (\iota_k):\iota_k\in I\}$ as its generating set of cofibrations and $\cS (J)=\{\cS (\zeta_k): \zeta_k\in J\}$ as its generating set of trivial cofibrations. The weak equivalences are the $\tt DG\cD A$-morphisms that induce an isomorphism in homology. The fibrations are the $\tt DG\cD A$-morphisms that are surjective in all positive degrees $p>0$.\end{theo}

Below, we will describe the cofibrations and functorial fibrant and cofibrant replacement functors.\medskip

The model structure on $\tt DG\cD A$ is obtained by Quillen transfer of the model structure on $\tt{DG\cD M}=\tt{Ch}_+(\cD)$. However, since $\cD$-modules (resp., $\cD$-algebras) are actually sheaves of modules (resp., sheaves of algebras), the category of differential graded $\cD$-modules (resp., differential graded $\cD$-algebras) over $X$, is rather ${\tt DG_+qcMod}(\cD_X)$ (resp., ${\tt DG_+qcCAlg}(\cD_X)$). In view of Proposition \ref{EquivShMod} (resp., Proposition \ref{EquivShAlg}), the finitely generated model structure on $\tt DG\cD M$ (resp., $\tt DG\cD A$) induces a finitely generated model structure on ${\tt DG_+qcMod}(\cD_X)$ (resp., ${\tt DG_+qcCAlg}(\cD_X)$).


\section{Description of $\tt DG\cD A$-cofibrations}

\subsection{Relative Sullivan $\cD$-algebras}

We recall the definition of relative Sullivan $\cD$-algebras \cite{BPP1}.\medskip

If $(A,d_A)\in\tt DG\cD A$ and if $(M,d_M)\in\tt DG\cD M$, then $(A\0\cS M,d)\in\tt DG\cD A$. The differential $d_S$ of $\cS M$ is canonically generated by $d_M$ and the differential $d$ of $A\0\cS M$ is given by \be\label{Split}d=d_A\0\id+\id\0\, d_S\;.\ee If $V\in\tt G\cD M$, we have $(V,0)\in\tt DG\cD M$ and $A\0\cS V\in\tt G\cD A$. In the sequel, we equip this graded $\cD$-algebra with a differential $d$ that coincides with $d_A\0\id$ on $A\0 1_\cO\simeq A$, but not with some differential $\id\0\, d_S$ on $1_A\0\cS V\simeq \cS V$. To distinguish such a differential graded $\cD$-algebra from $(A\0\cS V,d)$ with differential (\ref{Split}), we denote it by $(A\boxtimes\cS V,d)$.

\begin{defi} A {\bf relative Sullivan $\cD$-algebra} $(\,${\small RS$\cD\!$A}$\,)$ is a $\tt DG\cD A$-morphism
$$(A,d_A)\to(A\boxtimes \cS V,d)\;$$ that sends $a\in A$ to $a\0 1\in A\boxtimes \cS V$. Here $V$ is a free non-negatively graded $\mathcal{D}$-module, which admits a homogeneous basis $(g_\za)_{\za\in J}$ that is indexed by a well-ordered set $J$, and is such that \be\label{Lowering}d g_\za \in A\boxtimes \cS V_{<\za}\;,\ee for all $\za\in J$. In the last requirement, we set $V_{<\za}:=\bigoplus_{\zb<\za}\cD\cdot g_\zb\,$. We refer to Property (\ref{Lowering}) by saying that $d$ is {\bf lowering}. A {\small RS$\cD\!$A} with Property \be\label{minimal}\za\le\zb \Rightarrow \deg g_\za\le\deg g_\zb\;,\ee where $\deg g_\za$ is the degree of $g_\za$ $(\,$resp., with Property (\ref{Split}); over $(A,d_A)=(\cO,0)$$\,)$ is called a {\bf minimal} {\small RS$\cD\!$A} $(\,$resp., a {\bf split} {\small RS$\cD\!$A}; a {\bf Sullivan $\cD$-algebra} $(\,${\small S$\cD\!$A}$\,)$$\,)$.\end{defi}

The next lemma allows to define non-split {\small RS$\cD$A}-s, as well as ${\tt DG\cD A}$-morphisms from such an {\small RS$\cD$A} into another differential graded $\cD$-algebra.

\begin{lem}\label{LemRSA} Let $(T,d_T)\in\tt DG\cD A$, let $(g_j)_{j\in J}$ be a family of symbols of degree $n_j\in \N$, and let $V=\bigoplus_{j\in J}\cD\cdot g_j$ be the free non-negatively graded $\cD$-module with homogeneous basis $(g_j)_{j\in J}$.\smallskip

(i) To endow the graded $\cD$-algebra $T\0\cS V$ with a differential graded $\cD$-algebra structure $d$, it suffices to define \be\label{CondRSADiff}d g_j\in T_{n_j-1}\cap d_T^{-1}\{0\}\;,\ee to extend $d$ as $\cD$-linear map to $V$, and to equip $T\0\cS V$ with the differential $d$ given, for any $t\in T_p,\,v_1\in V_{n_1},\,\ldots,\,v_k\in V_{n_k}\,$, by \be\label{DefRSADiff}d({t}\0 v_1\odot\ldots\odot v_k)=$$ $$d_T({t})\0 v_1\odot\ldots\odot v_k+(-1)^p\sum_{\ell=1}^k(-1)^{n_\ell\sum_{j<\ell}n_j}({t}\ast d(v_\ell))\0v_1\odot\ldots\widehat{\ell}\ldots\odot v_k\;,\ee where $\ast$ is the multiplication in $T$. If $J$ is a well-ordered set, the natural map $$(T,d_T)\ni {t}\mapsto {t}\0 1_\cO\in (T\boxtimes\cS V,d)$$ is a {\small RS$\cD\!$A}.\smallskip

(ii) Moreover, if $(B,d_B)\in{\tt DG\cD A}$ and $p\in{\tt DG\cD A}(T,B)$, it suffices -- to define a morphism $q\in{\tt DG\cD A}(T\boxtimes\cS V,B)$ (where the differential graded $\cD$-algebra $(T\boxtimes\cS V,d)$ is constructed as described in (i)) -- to define \be\label{CondRSAMorph}q(g_j)\in B_{n_j}\cap d_B^{-1}\{p\,d(g_j)\}\;,\ee to extend $q$ as $\cD$-linear map to $V$, and to define $q$ on $T\0\cS V$ by \be\label{DefRSAMorph}q({t}\0 v_1\odot\ldots\odot v_k)=p({t})\star q(v_1)\star\ldots\star q(v_k)\;,\ee where $\star$ denotes the multiplication in $B$.\end{lem}

The reader might consider that the definition of $d(t\0 f)$, $f\in\cO$, is not an edge case of Definition (\ref{DefRSADiff}); if so, it suffices to add the definition $d(t\0 f)=d_T(t)\0 f\,.$ Note also that Definition (\ref{DefRSADiff}) is the only possible one. Indeed, denote the multiplication in $T\0\cS V$ (see Equation (13) in \cite{BPP1}) by $\diamond$ and choose, to simplify, $k=2$. Then, if $d$ is any differential that is compatible with the graded $\cD$-algebra structure of $T\0\cS V$, and coincides with $d_T(t)\0 1_\cO\simeq d_T(t)$ on any $t\0 1_\cO\simeq t\in T$ (since $(T,d_T)\to (T\boxtimes\cS V,d)$ must be a $\tt DG\cD A$-morphism) and with $d(v)\0 1_\cO\simeq d(v)$ on any $1_T\0 v\simeq v\in V$ (since $d(v)\in T$), we have necessarily
\bea & d(t\0 v_1\odot v_2)=\eea
\bea & d(t\0 1_\cO)\diamond (1_{T}\0 v_1)\diamond (1_{T}\0 v_2)+\\ &(-1)^p(t\0 1_\cO)\diamond d(1_{T}\0 v_1)\diamond (1_{T}\0 v_2)+\\ &(-1)^{p+n_1}(t\0 1_\cO)\diamond (1_{T}\0 v_1)\diamond d(1_{T}\0 v_2)=\eea

$$(d_T(t)\0 1_\cO)\diamond (1_{T}\0 v_1)\diamond (1_{T}\0 v_2)+$$ $$(-1)^p(t\0 1_\cO)\diamond (d(v_1)\0 1_\cO)\diamond (1_{T}\0 v_2)+$$ $$(-1)^{p+n_1}(t\0 1_\cO)\diamond (1_{T}\0 v_1)\diamond (d(v_2)\0 1_\cO)=$$

$$d_T(t)\0 v_1\odot v_2 + (-1)^p(t\ast d(v_1))\0 v_2+(-1)^{p+n_1n_2}(t\ast d(v_2))\0 v_1\;.$$
An analogous remark holds for Definition (\ref{DefRSAMorph}).

\begin{proof} It is easily checked that the {\small RHS} of Equation (\ref{DefRSADiff}) is graded symmetric in its arguments $v_i$ and $\cO$-linear with respect to all arguments. Hence, the map $d$ is a degree $-1$ $\cO$-linear map that is well-defined on $T\0\cS V$. To show that $d$ endows $T\0\cS V$ with a differential graded $\cD$-algebra structure, it remains to prove that $d$ squares to 0, is $\cD$-linear and is a graded derivation for $\diamond$. The last requirement follows immediately from the definition, for $\cD$-linearity it suffices to prove linearity with respect to the action of vector fields -- what is a straightforward verification --, whereas 2-nilpotency is a consequence of Condition (\ref{CondRSADiff}). The proof of (ii) is similar. \end{proof}

We are now prepared to give an example of a minimal non-split {\small RS$\cD$A}.

\begin{ex}\label{Lem1}\emph{Consider the generating cofibrations $\iota_n:S^{n-1}\to D^n$, $n\ge 1$, and $\iota_0:0\to S^0$ of the model structure of $\tt DG\cD M$. The {\it pushouts} of the induced generating cofibrations $$\psi_n=\cS(\iota_n)\quad\text{and}\quad \psi_0=\cS(\iota_0)$$ of the transferred model structure on $\tt DG\cD A$ are important instances of minimal non-split {\small RS$\cD$A}-s -- see Figure 2 and Equations (\ref{kappa}), (\ref{RSA-d}), (\ref{i}), (\ref{Cond2}), and (\ref{j}).} \end{ex}

\begin{proof} We first consider a pushout diagram for $\psi:=\psi_n$, for $n\ge 1$: see Figure \ref{PDiag},
\begin{figure}[h]
\begin{center}
\begin{tikzpicture}
  \matrix (m) [matrix of math nodes, row sep=3em, column sep=3em]
    {  \cS(S^{n-1}) & (T,d_T)  \\
       \cS(D^n) & \\ };
 \path[->]
 (m-1-1) edge  node[above] {$\scriptstyle{\zf}$} (m-1-2);
  \path[->]
 (m-1-1) edge  node[left] {$\scriptstyle{\psi}$} (m-2-1);
\end{tikzpicture}
\end{center}
\caption{Pushout diagram}\label{PDiag}
\end{figure}
where $(T,d_T)\in \tt DG\mathcal{D}A$ and where $\phi:(\cS(S^{n-1}),0)\to(T,d_T)$ is a $\tt DG\cD A$-morphism. \medskip

In the following, the generator of $S^{n-1}$ (resp., the generators of $D^n$) will be denoted by $1_{n-1}$ (resp., by $\mathbb{I}_n$ and $s^{-1}\mathbb{I}_n$, where $s^{-1}$ is the desuspension operator).\medskip

Note that, since $\cS(S^{n-1})$ is the free {\small DG$\cD$A} over the {\small DG$\cD$M} $S^{n-1}$, the $\tt DG\cD A$-morphism $\phi$ is uniquely defined by the $\tt DG\mathcal{D}M$-morphism $\phi|_{S^{n-1}}: S^{n-1}\to \op{For}(T,d_T)$, where $\op{For}$ is the forgetful functor. Similarly, since $S^{n-1}$ is, as {\small G$\cD$M}, free over its generator $1_{n-1}$, the restriction $\phi|_{S^{n-1}}$ is, as $\tt G\cD M$-morphism, completely defined by its value $\phi(1_{n-1})\in T_{n-1}$. The map $\phi|_{S^{n-1}}$ is then a $\tt DG\cD M$-morphism if and only if we choose \be\label{kappa}\zk_{n-1}:=\phi(1_{n-1})\in\ker_{n-1}d_T\;.\ee

We now define the pushout of $(\psi,\phi)$: see Figure \ref{CPD}. 
\begin{figure}[h]
\begin{center}
\begin{tikzpicture}
  \matrix (m) [matrix of math nodes, row sep=3em, column sep=3em]
    {  \cS(S^{n-1}) & (T,d_T)  \\
       \cS(D^n) & (T\boxtimes\cS(S^n),d)  \\ };
 \path[->]
 (m-1-2) edge  node[right] {$\scriptstyle{i}$} (m-2-2);
 \path[->]
 (m-1-1) edge  node[above] {$\scriptstyle{\zf}$} (m-1-2);
  \path[->]
 (m-1-1) edge  node[left] {$\scriptstyle{\psi}$} (m-2-1);
  \path[->]
 (m-2-1) edge  node[above] {$\scriptstyle{j}$} (m-2-2);
\end{tikzpicture}
\caption{Completed pushout diagram}\label{CPD}
\end{center}
\end{figure}
\noindent In the latter diagram, the differential $d$ of the {\small G$\cD$A} $T\boxtimes\cS(S^n)$ is defined as described in Lemma \ref{LemRSA}. Indeed, we deal here with the free non-negatively graded $\cD$-module $S^n=S^n_n=\cD\cdot 1_n$ and set $$d(1_n):=\zk_{n-1}=\zf(1_{n-1})\in\ker_{n-1}d_T\;.$$ Hence, if $x_\ell\simeq x_\ell\cdot 1_n\in\cD\cdot 1_n$, we get $d(x_\ell)=x_\ell\cdot\zk_{n-1}$, and, if $t\in T_p$, we obtain \be\label{RSA-d}d({t}\0 x_1\odot\ldots\odot x_k)=$$ $$d_T({t})\0 x_1\odot\ldots\odot x_k+(-1)^p\sum_{\ell=1}^k(-1)^{n(\ell-1)}({t}\ast (x_\ell\cdot\zk_{n-1}))\0 x_1\odot\ldots\widehat{\ell}\ldots\odot x_k\;,\ee see Equation (\ref{DefRSADiff}). Eventually the map \be\label{i}i:(T,d_T)\ni t\mapsto t\0 1_\cO\in (T\boxtimes\cS(S^n),d)\ee is a (minimal and non-split) {\small RS$\cD$A}.\medskip

Just as $\phi$, the $\tt DG\cD A$-morphism $j$ is completely defined if we define it as $\tt DG\cD M$-morphism on $D^n$. The choices of $j(\mathbb{I}_n)$ and $j(s^{-1}\mathbb{I}_{n})$ define $j$ as $\tt G\cD M$-morphism. The commutation condition of $j$ with the differentials reads \be\label{Cond1}j(s^{-1}\mathbb{I}_n)=d\,j(\mathbb{I}_n)\;:\ee only $j(\mathbb{I}_n)$ can be chosen freely in $(T\0\cS(S^n))_n\,$.\medskip

The diagram of Figure \ref{CPD} is now fully described. To show that it commutes, observe that, since the involved maps $\phi,i,\psi$, and $j$ are all $\tt DG\cD A$-morphisms, it suffices to check commutation for the arguments $1_\cO$ and $1_{n-1}$. Only the second case is non-obvious; we get the condition \be\label{Cond2}d\,j(\mathbb{I}_n)=\zk_{n-1}\0 1_\cO\;.\ee It is easily seen that the unique solution is \be\label{j}j(\mathbb{I}_n)=1_T\0 1_n\in(T\0\cS(S^{n}))_n\;.\ee

To prove that the commuting diagram of Figure \ref{CPD} is the searched pushout, it now suffices to prove its universality. Therefore, take $(B,d_B)\in\tt DG\cD A$, as well as two $\tt DG\cD A$-morphisms $i':(T,d_T)\to (B,d_B)$ and $j':\cS(D^n)\to(B,d_B)$, such that $j'\circ \psi=i'\circ\phi$, and show that there is a unique $\tt DG\mathcal{D}A$-morphism $\chi:(T\boxtimes \cS(S^n),d)\to (B,d_B)$, such that $\chi\circ i=i'$ and $\chi\circ j=j'$.\medskip

If $\chi$ exists, we have necessarily $$\chi(t\0 x_1\odot\ldots\odot x_k)=\chi((t\0 1_\cO)\diamond (1_T\0 x_1)\diamond \ldots\diamond (1_T\0 x_k))$$ \be\label{UP1}=\chi(i(t))\star \chi(1_T\0 x_1)\star\ldots\star \chi(1_T\0 x_k)\;,\ee where we used the same notation as above. Since any differential operator $x_i\simeq x_i\cdot 1_n$ is generated by functions and vector fields, we get \be\label{UP2}\chi(1_T\0 x_i)=\chi(1_T\0 x_i\cdot 1_n)=x_i\cdot \chi(1_T\0 1_n)=x_i\cdot \chi(j(\mathbb{I}_n))=x_i\cdot j'(\mathbb{I}_n)=j'(x_i\cdot\mathbb{I}_n)\;.\ee When combining (\ref{UP1}) and (\ref{UP2}), we see that, if $\chi$ exists, it is necessarily defined by \be\label{UP3}\chi(t\0 x_1\odot\ldots\odot x_k)=i'(t)\star j'(x_1\cdot\mathbb{I}_n)\star\ldots\star j'(x_k\cdot\mathbb{I}_n)\;.\ee This solves the question of uniqueness.\medskip

We now convince ourselves that (\ref{UP3}) defines a $\tt DG\cD A$-morphism $\chi$ (let us mention explicitly that we set in particular $\chi(t\0 f)=f\cdot i'(t)$, if $f\in\cO$). It is straightforwardly verified that $\chi$ is a well-defined $\cD$-linear map of degree 0 from $T\0\cS(S^n)$ to $B$, which respects the multiplications and the units. The interesting point is the chain map property of $\chi$. Indeed, consider, to simplify, the argument $t\0 x$, what will disclose all relevant insights. Assume again that $t\in T_p$ and $x\in S^n$, and denote the differential of $\cS(D^n)$, just as its restriction to $D^n$, by $s^{-1}$. It follows that $$d_B(\chi(t\0 x))=i'(d_T(t))\star j'(x\cdot\mathbb{I}_n)+(-1)^{p}\,i'(t)\star j'(x\cdot s^{-1}\mathbb{I}_n)\;.$$ Since $\psi(1_{n-1})=s^{-1}\mathbb{I}_n$ and $j'\circ\psi=i'\circ\phi$, we obtain $j'(s^{-1}\mathbb{I}_n)=i'(\zf(1_{n-1}))=i'(\zk_{n-1})$. Hence, $$d_B(\chi(t\0 x))=\chi(d_T(t)\0 x)+(-1)^{p}\,i'(t)\star i'(x\cdot\zk_{n-1})=$$ $$\chi(d_T(t)\0 x+(-1)^{p}t\ast(x\cdot\zk_{n-1}))=\chi(d(t\0 x))\;.$$ As afore-mentioned, no new feature appears, if we replace $t\0 x$ by a general argument.\medskip

As the conditions $\chi\circ i=i'$ and $\chi\circ j=j'$ are easily checked, this completes the proof of the statement that any pushout of any $\psi_n$, $n\ge 1$, is a minimal non-split {\small RS$\cD$A}.\medskip

The proof of the similar claim for $\psi_0$ is analogous and even simpler, and will not be detailed here.\end{proof}

Actually pushouts of $\psi_0$ are border cases of pushouts of the $\psi_n$-s, $n\ge 1$. In other words, to obtain a pushout of $\psi_0$, it suffices to set, in Figure \ref{CPD} and in Equation (\ref{RSA-d}), the degree $n$ to 0. Since we consider exclusively non-negatively graded complexes, we then get $\cS(S^{-1})=\cS(0)=\cO$, $\cS(D^0)=\cS(S^0)$, and $\zk_{-1}=0$.

\subsection{$\tt DG\cD A$-cofibrations}

The following theorem characterizes the cofibrations of the cofibrantly generated model structure we constructed on $\tt DG\mathcal{D}A$.

\begin{theo}\label{Cof} The $\tt DG\mathcal{D}A$-cofibrations are exactly the retracts of the relative Sullivan $\cD$-algebras. \end{theo}

We first prove the following lemma.

\begin{lem} The $\tt DG\mathcal{D}A$-cofibrations are exactly the retracts of the transfinite compositions of pushouts of generating cofibrations $$\psi_n:\cS(S^{n-1})\to\cS(D^n),\quad n\ge 0\;.$$\end{lem}

\begin{proof} For concise additional information on model categories, we refer to \cite[Appendices 8.4 and 8.6]{BPP1}.

In any cofibrantly generated model category $\tt M$ with generating cofibrations $I$, every cofibration is a retract of an $I$-cell \cite[Proposition 2.1.18]{Hov}. Moreover, in view of \cite[Lemma 2.1.10]{Hov}, we have \be\label{SpecCof}I\text{-}\op{cell}\subset\op{LLP}(\op{RLP}(I))=\op{Cof}\;.\ee Since cofibrations are closed under retracts, it follows that any retract of an $I$-cell is a cofibration. Hence, cofibrations are exactly the retracts of the $I$-cells, i.e., the retracts of the transfinite compositions of pushouts of elements of $I$. For $\tt M = DG\cD A$, we thus find that the cofibrations are the retracts of the transfinite compositions of pushouts of $\psi_n$-s, $n\ge 0\,.$\end{proof}

The proof of Theorem \ref{Cof} thus reduces to the proof of

\begin{theo}\label{Reduction} The transfinite compositions of pushouts of $\psi_n$-s, $n\ge 0$, are exactly the relative Sullivan $\cD$-algebras.\end{theo}

\begin{lem}\label{DirSumTenPro} For any $M,N\in\tt DG\cD M$, we have $$\cS(M\oplus N)\simeq \cS M\0 \cS N\;$$ in $\tt DG\cD A\,.$\end{lem}

\begin{proof} 
It 
suffices to remember that the binary coproduct in the category $\tt DG\cD M=Ch_+(\cD)$ (resp., the category $\tt DG\cD A=CMon(DG\cD M)$) of non-negatively graded chain complexes of $\cD$-modules (resp., the category of commutative monoids in $\tt DG\cD M$) is the direct sum (resp., the tensor product). The conclusion then follows from the facts that $\cS$ is the left adjoint of the forgetful functor and that any left adjoint commutes with colimits.\end{proof}

Any ordinal is zero, a successor ordinal, or a limit ordinal. We denote the class of all successor ordinals (resp., all limit ordinals) by $\frak{O}_s$ (resp., $\frak{O}_\ell$).

\begin{proof}[Proof of Theorem \ref{Reduction}] (i) Consider an ordinal $\zl$ and a $\zl$-sequence in $\tt DG\cD A$, i.e., a colimit respecting functor $X:\zl\to \tt DG\cD A$ (here $\zl$ is viewed as the category whose objects are the ordinals $\za<\zl$ and which contains a unique morphism $\za\to\zb$ if and only if $\za\le\zb$): $$X_0\to X_1\to \ldots \to X_n\to X_{n+1}\to\ldots X_\zw\to X_{\zw+1}\to \ldots \to X_\za\to X_{\za+1}\to \ldots$$ We assume that, for any $\za$ such that $\za+1<\zl$, the morphism $X_\za\to X_{\za+1}$ is a pushout of some $\psi_{n_{\za+1}}$ ($n_{\za+1}\ge 0$). Then the morphism $X_0\to \op{colim}_{\za<\zl}X_\za$ is exactly what we call a transfinite composition of pushouts of $\psi_n$-s. Our task is to show that this morphism is a {\small RS$\cD$A}.\medskip

We first compute the terms $X_\za$, $\za<\zl,$ of the $\zl$-sequence, then we determine its colimit. For $\za<\zl$ (resp., for $\za<\zl, \za\in{\frak O}_s$), we denote the differential graded $\cD$-algebra $X_\za$ (resp., the $\tt DG\cD A$-morphism $X_{\za-1}\to X_{\za}$) by $(A_\za,d_\za)$ (resp., by $X_{\za,\za-1}:(A_{\za-1},d_{\za-1})\to (A_{\za},d_{\za})$). Since $X_{\za,\za-1}$ is the pushout of some $\psi_{n_{\za}}$ and some $\tt DG\cD A$-morphism $\zf_{\za}$, its target algebra is of the form \be\label{Successor}(A_{\za},d_{\za})=(A_{\za-1}\boxtimes\cS\langle a_{\za}\ra,d_{\za})\;\ee and $X_{\za,\za-1}$ is the canonical inclusion \be\label{SuccessorMorph}X_{\za,\za-1}:(A_{\za-1},d_{\za-1})\ni \frak{a}_{\za-1}\mapsto \frak{a}_{\za-1}\0 1_\cO\in (A_{\za-1}\boxtimes\cS\langle a_\za\rangle,d_\za)\;,\ee see Example \ref{Lem1}. Here $a_{\za}$ is the generator $1_{n_{\za}}$ of $S^{n_{\za}}$ and $\langle a_{\za}\ra$ is the free non-negatively graded $\cD$-module $S^{n_{\za}}=\cD\cdot a_{\za}$ concentrated in degree $n_{\za}$; further, the differential \be\label{RSA-d-2}d_{\za}\;\;\text{is defined by (\ref{RSA-d}) from}\;\; d_{\za-1}\;\;\text{and}\;\;\zk_{n_{\za}-1}:=\zf_{\za}(1_{n_{\za}-1})\;.\ee In particular, $A_1=A_0\boxtimes\cS\langle a_1\ra\,,$ $d_1(a_1)=\zk_{n_1-1}=\zf_1(1_{n_1-1})\in A_0\,,$ and $X_{10}:A_0\to A_1$ is the inclusion.\medskip

\begin{lem}\label{Lem4} For any $\za<\zl$, we have
\be\label{Lem41}A_{\za}\simeq A_0\otimes \cS \langle a_\delta: \delta\leq\za, \zd\in\frak{O}_s\rangle\;\ee as a graded $\cD$-algebra,
and
\be\label{Lem42}d_{\za}(a_{\delta})\in A_0\otimes \cS \langle  a_\ze: \ze< \delta, \ze\in\frak{O}_s\rangle\;,\ee
for all $\delta\leq \za$, $\zd\in\frak{O}_s$. Moreover, for any $\zg\le\zb\le \za <\zl$, we have $$A_\zb=A_\zg\0\cS\langle a_\zd:\zg<\zd\le\zb,\zd\in{\frak O}_s\rangle$$ and the $\tt DG\cD A$-morphism $X_{\zb\zg}$ is the natural inclusion \be\label{Lem43}X_{\zb\zg}: (A_\zg,d_\zg)\ni \frak{a}_\zg\mapsto \frak{a}_\zg\0 1_\cO\in (A_\zb,d_\zb)\;.\ee Since the latter statement holds in particular for $\zg=0$ and $\zb=\za$, the $\tt DG\cD A$-inclusion $X_{\za 0}:(A_0,d_0)\to (A_\za,d_\za)$ is a {\small RS$\cD$A} $(\,$for the natural ordering of $\{a_\delta: \zd\le\za,\zd\in\frak{O}_s\}\,).$\end{lem}

\begin{proof}[Proof of Lemma \ref{Lem4}] To prove that this claim (i.e., Equations (\ref{Lem41}) -- (\ref{Lem43})) is valid for all ordinals that are smaller than $\zl$, we use a transfinite induction. Since the assertion obviously holds for $\za=1,$ it suffices to prove these properties for $\za<\zl$, assuming that they are true for all $\zb<\za$. We distinguish (as usually in transfinite induction) the cases $\za\in\frak{O}_s$ and $\za\in\frak{O}_\ell$.\medskip

If $\za\in\frak{O}_s$, it follows from Equation (\ref{Successor}), from the induction assumption, and from Lemma \ref{DirSumTenPro}, that $$A_\za=A_{\za-1}\0\cS\langle a_\za\ra\simeq A_0\0 \cS\langle a_\zd: \zd\le \za,\zd\in\frak{O}_s\ra\;,$$ as graded $\cD$-algebra. Further, in view of Equation (\ref{RSA-d-2}) and the induction hypothesis, we get $$d_\za(a_\za)=\zf_{\za}(1_{n_\za-1})\in A_{\za-1}=A_0\0 \cS\langle a_\zd: \zd<\za,\zd\in\frak{O}_s\ra\;,$$ and, for $\zd\le\za-1$, $\zd\in\frak{O}_s$, $$d_\za(a_\zd)=d_{\za-1}(a_\zd)\in A_0\otimes \cS \langle  a_\zg: \zg< \delta, \zg\in\frak{O}_s\rangle\;.$$ Finally, as concerns $X_{\zb\zg}$, the unique case to check is $\zg\le\za-1$ and $\zb=\za$. The $\tt DG\cD A$-map $X_{\za-1,\zg}$ is an inclusion $$X_{\za-1,\zg}: A_\zg\ni \frak{a}_\zg\mapsto \frak{a}_\zg\0 1_\cO\in A_{\za-1}\;$$ (by induction), and so is the $\tt DG\cD A$-map $$X_{\za,\za-1}:A_{\za-1}\ni \frak{a}_{\za-1}\mapsto \frak{a}_{\za-1}\0 1_\cO\in A_\za\;$$ (in view of (\ref{SuccessorMorph})). The composite $X_{\za\zg}$ is thus a $\tt DG\cD A$-inclusion as well.\medskip

In the case $\za\in\frak{O}_\ell$, i.e., $\za=\op{colim}_{\zb<\za}\zb$, we obtain $(A_\za,d_\za)=\op{colim}_{\zb<\za}(A_\zb,d_\zb)$ in $\tt DG\cD A$, since $X$ is a colimit respecting functor. The index set $\za$ is well-ordered, hence, it is a directed poset. Moreover, for any $\zd\le\zg\le\zb<\za$, the $\tt DG\cD A$-maps $X_{\zb\zd}$, $X_{\zg\zd}$, and $X_{\zb\zg}$ satisfy $X_{\zb\zd}=X_{\zb\zg}\circ X_{\zg\zd}\,$. It follows that the family $(A_\zb,d_\zb)_{\zb<\za},$ together with the family $X_{\zb\zg}$, $\zg\le\zb<\za$, is a direct system in $\tt DG\cD A$, whose morphisms are, in view of the induction assumption, natural inclusions $$X_{\zb\zg}:A_\zg\ni \frak{a}_\zg\mapsto \frak{a}_\zg\0 1_\cO\in A_\zb\;.$$ The colimit $(A_\za,d_\za)=\op{colim}_{\zb<\za}(A_\zb,d_\zb)$ is thus a direct limit. We proved in \cite{BPP1} that a direct limit in $\tt DG\cD A$ coincides with the corresponding direct limit in $\tt DG\cD M$, or even in $\tt Set$ (which is then naturally endowed with a differential graded $\cD$-algebra structure). As a set, the direct limit $(A_\za,d_\za)=\op{colim}_{\zb<\za}(A_\zb,d_\zb)$ is given by $$A_\za=\coprod_{\zb<\za}A_\zb/\sim\;,$$ where $\sim$ means that we identify $\frak{a}_\zg$, $\zg\le\zb$, with $$\frak{a}_\zg\sim X_{\zb\zg}(\frak{a}_\zg)=\frak{a}_\zg\0 1_\cO\;,$$ i.e., that we identify $A_\zg$ with $$A_\zg\sim A_\zg\0\cO\subset A_\zb\;.$$ It follows that $$A_\za=\bigcup_{\zb<\za}A_\zb=A_0\0\cS\langle a_\zd:\zd<\za,\zd\in\frak{O}_s\ra=A_0\0\cS\langle a_\zd:\zd\le\za,\zd\in\frak{O}_s\ra\;.$$ As just mentioned, this set $A_\za$ can naturally be endowed with a differential graded $\cD$-algebra structure. For instance, since, in view of what has been said, all $\sim\,$-$\,$classes consist of a single element, and since any $\frak{a}_\za\in A_\za$ belongs to some $A_\zb$, $\zb<\za$, the differential $d_\za$ is defined by $d_\za(\frak{a}_\za)=d_\zb(\frak{a}_\za)$. In particular, any generator $a_\zd$, $\zd\le\za$, $\zd\in\frak{O}_s$, belongs to $A_\zd$. Hence, by definition of $d_\za$ and in view of the induction assumption, we get $$d_\za(a_\zd)=d_\zd(a_\zd)\in A_0\0\cS\langle a_\ze:\ze<\zd,\ze\in\frak{O}_s\ra\;.$$ Eventually, since $X$ is colimit respecting, not only $A_\za =\colim_{\zb<\za}A_\zb=\bigcup_{\zb<\za}A_\zb$, but, furthermore, for any $\zg<\za$, the $\tt DG\cD A$-morphism $X_{\za\zg}:A_\zg\to A_\za$ is the map $X_{\za\zg}:A_\zg\to \bigcup_{\zb<\za}A_\zb$, i.e., the canonical inclusion.\end{proof}

We now come back to the proof of Part (i) of Theorem \ref{Reduction}, i.e., we now explain why the morphism $i:(A_0,d_0)\to C$, where $C=\op{colim}_{\za<\zl}(A_\za,d_\za)$ and where $i$ is the first of the morphisms that are part of the colimit construction, is a {\small RS$\cD$A} -- see above. If $\zl\in\frak{O}_s$, the colimit $C$ coincides with $(A_{\zl-1},d_{\zl-1})$ and $i=X_{\zl-1,0}$. Hence, the morphism $i$ is a {\small RS$\cD$A} in view of Lemma \ref{Lem4}. If $\zl\in\frak{O}_\ell$, the colimit $C=\op{colim}_{\za<\zl}(A_\za,d_\za)$ is, like above, the direct limit of the direct $\tt DG\cD A$-system $(X_\za=(A_\za,d_\za),X_{\za\zb})$ indexed by the directed poset $\zl$, whose morphisms $X_{\za\zb}$ are, in view of Lemma \ref{Lem4}, canonical inclusions. Hence, $C$ is again an ordinary union: \be\label{LimUnion}C=\bigcup_{\za<\zl}A_\za=A_0\0\cS\langle a_\zd: \zd<\zl,\zd\in\frak{O}_s\ra\;,\ee where the last equality is due to Lemma \ref{Lem4}. We define the differential $d_C$ on $C$ exactly as we defined the differential $d_\za$ on the direct limit in the proof of Lemma \ref{Lem4}. It is then straightforwardly checked that $i$ is a {\small RS$\cD$A}.\medskip

(ii) We still have to show that any {\small RS$\cD$A} $(A_0,d_0)\to (A_0\boxtimes\cS V,d)$ can be constructed as a transfinite composition of pushouts of generating cofibrations $\psi_n$, $n\ge 0$. Let $(a_j)_{j\in J}$ be the basis of the free non-negatively graded $\cD$-module $V$. Since $J$ is a well-ordered set, it is order-isomorphic to a unique ordinal $\zm=\{0,1,\ldots,n,\ldots,\zw,\zw+1,\ldots\}$, whose elements can thus be utilized to label the basis vectors. However, we prefer using the following order-respecting relabelling of these vectors: $$a_0\rightsquigarrow a_1, a_1\rightsquigarrow a_2,\ldots, a_n\rightsquigarrow a_{n+1},\ldots, a_\omega\rightsquigarrow a_{\omega+1}, a_{\omega+1}\rightsquigarrow a_{\omega+2},\ldots$$ In other words, the basis vectors of $V$ can be labelled by the successor ordinals that are strictly smaller than $\zl:=\zm+1\,$ (this is true, whether $\zm\in\frak{O}_s$, or $\zm\in\frak{O}_\ell\,$): $$V=\bigoplus_{\zd<\zl,\;\zd\in\frak{O}_s} \cD\cdot a_\zd\;.$$

For any $\za<\zl$, we now set $$(A_\za,d_\za):=(A_0\boxtimes \cS\langle a_\delta: \delta\leq\za, \zd\in\frak{O}_s\rangle,d|_{A_\za})\;.$$ It is clear that $A_\za$ is a graded $\cD$-subalgebra of $A_0\0\cS V$. Since $A_\za$ is generated, as an algebra, by the elements of the types $\frak{a}_0\0 1_\cO$ and $D\cdot(1_{A_0}\0 a_\zd)$, $D\in\cD$, $\zd\le\za,$ $\zd\in\frak{O}_s$, and since $$d(\frak{a}_0\0 1_\cO)=d_0(\frak{a}_0)\0 1_\cO\in A_\za$$ and $$d(D\cdot(1_{A_0}\0 a_\zd))\in A_0\0\cS\langle a_\ze:\ze<\zd,\ze\in\frak{O}_s\ra\subset A_\za\;,$$ the derivation $d$ stabilizes $A_\za$. Hence, $(A_\za,d_\za)=(A_\za,d|_{A_\za})$ is actually a differential graded $\cD$-subalgebra of $(A_0\boxtimes\cS V,d)$.\medskip

If $\zb\le\za<\zl$, the algebra $(A_\zb,d|_{A_\zb})$ is a differential graded $\cD$-subalgebra of $(A_\za,d|_{A_\za})$, so that the canonical inclusion $i_{\za\zb}:(A_\zb,d_\zb)\to(A_\za,d_\za)$ is a $\tt DG\cD A$-morphism. In view of the techniques used in (i), it is obvious that the functor $X=(A_-,d_-):\zl\to \tt DG\cD A$ respects colimits, and that the colimit of the whole $\zl$-sequence (remember that $\zl=\zm+1\in\frak{O}_s$) is the algebra $(A_\zm,d_\zm)=(A_0\boxtimes\cS V,d)$, i.e., the original algebra.\medskip

The {\small RS$\cD$A} $(A_0,d_0)\to (A_0\boxtimes\cS V,d)$ has thus been built as transfinite composition of canonical $\tt DG\cD A$-inclusions $i:(A_{\za},d_\za)\to (A_{\za+1},d_{\za+1})$, $\za+1<\zl$. Recall that $$A_{\za+1}=A_\za\0\cS\langle a_{\za+1}\ra\simeq A_\za\0\cS(S^n)\;,$$ if we set $n:=\deg(a_{\za+1})$. It suffices to show that $i$ is a pushout of $\psi_n$, see Figure \ref{CPD2}.
\begin{figure}[h]
\begin{center}
\begin{tikzpicture}
  \matrix (m) [matrix of math nodes, row sep=3em, column sep=3em]
    {  \cS(S^{n-1}) & (A_\za,d_\za)  \\
       \cS(D^n) & (A_\za\boxtimes\cS(S^n),d_{\za+1})  \\ };
 \path[->]
 (m-1-2) edge  node[right] {$\scriptstyle{i}$} (m-2-2);
 \path[->]
 (m-1-1) edge  node[above] {$\scriptstyle{\zf}$} (m-1-2);
  \path[->]
 (m-1-1) edge  node[left] {$\scriptstyle{\psi_n}$} (m-2-1);
  \path[->]
 (m-2-1) edge  node[above] {$\scriptstyle{j}$} (m-2-2);
\end{tikzpicture}
\caption{$i$ as pushout of $\psi_n$}\label{CPD2}
\end{center}
\end{figure}
\noindent We will detail the case $n\ge 1$. Since all the differentials are restrictions of $d$, we have $\zk_{n-1}:=d_{\za+1}(a_{\za+1})\in A_\za\cap\ker_{n-1}d_{\za}$, and $\zf(1_{n-1}):=\zk_{n-1}$ defines a $\tt DG\cD A$-morphism $\zf$, see Example \ref{Lem1}. When using the construction described in Example \ref{Lem1}, we get the pushout $i:(A_\za,d_\za)\to (A_\za\boxtimes\cS(S^n),\p)$ of the morphisms $\psi_n$ and $\zf$. Here $i$ is the usual canonical inclusion and $\p$ is the differential defined by Equation (\ref{RSA-d}). It thus suffices to check that $\p=d_{\za+1}$. Let $\frak{a}_\za\in A^p_\za$ and let $x_1\simeq x_1\cdot\, a_{\za+1},\ldots,x_k\simeq x_k\cdot\, a_{\za+1}\in\cD\cdot\, a_{\za+1}=S^n$. Assume, to simplify, that $k=2$; the general case is similar. When denoting the multiplication in $A_\za$ (resp., $A_{\za+1}=A_\za\0\cS(S^n)$) as usual by $\ast$ (resp., $\star\,$), we obtain $$\p(\frak{a}_\za\0 x_1\odot x_2)=$$
$$d_\za(\frak{a}_\za)\0 x_1\odot x_2 + (-1)^p(\frak{a}_\za\ast (x_1\cdot \zk_{n-1}))\0 x_2+(-1)^{p+n}(\frak{a}_\za\ast (x_2\cdot\zk_{n-1}))\0 x_1=$$

$$(d_\za(\frak{a}_\za)\0 1_\cO)\star (1_{A_\za}\0 x_1)\star (1_{A_\za}\0 x_2)+$$ $$(-1)^p(\frak{a}_\za\0 1_\cO)\star ((x_1\cdot\zk_{n-1})\0 1_\cO)\star (1_{A_\za}\0 x_2)+$$ $$(-1)^{p+n}(\frak{a}_\za\0 1_\cO)\star (1_{A_\za}\0 x_1)\star ((x_2\cdot\zk_{n-1})\0 1_\cO)=$$

$$d_{\za+1}(\frak{a}_\za\0 1_\cO)\star (1_{A_\za}\0 x_1)\star (1_{A_\za}\0 x_2)+$$ $$(-1)^p(\frak{a}_\za\0 1_\cO)\star d_{\za+1}(1_{A_\za}\0 x_1)\star (1_{A_\za}\0 x_1)+$$ $$(-1)^{p+n}(\frak{a}_\za\0 1_\cO)\star (1_{A_\za}\0 x_1)\star d_{\za+1}(1_{A_\za}\0 x_2)=$$

$$d_{\za+1}(\frak{a}_\za\0 x_1\odot x_2)\;.$$
\end{proof}


\section{Explicit functorial cofibration -- fibration decompositions}\label{Factorizations}

\def\fb{{\mathfrak b}}

In \cite[Theorem 4]{BPP1}, we proved that any $\tt DG\cD A$-morphism $\zf:A\to B$ admits a functorial factorization \be\label{TrivCofFib}A\stackrel{i}{\longrightarrow}A\0\cS U\stackrel{p}{\longrightarrow}B\;,\ee where $p$ is a fibration and $i$ is a weak equivalence, as well as a split minimal {\small RS$\cD$A}. In view of Theorem \ref{Cof} of the present paper, the morphism $i$ is thus a cofibration, with the result that we actually constructed a natural decomposition $\zf=p\circ i$ of an arbitrary $\tt DG\cD A$-morphism $\zf$ into $i\in\text{\small TrivCof}$ and $p\in\text{\small Fib}$. The description of this factorization is summarized below, in Theorem \ref{P:c-tf_tc-f}, which provides essentially an explicit natural `{\small Cof -- TrivFib}' decomposition \be\label{CofTrivFib}A\stackrel{i'}{\longrightarrow}A\0\cS U'\stackrel{p'}{\longrightarrow}B\;.\ee

Since the model category $\tt DG\cD A$ is cofibrantly generated with generating cofibrations (resp., trivial cofibrations) $\cS(I)$ (resp., $\cS(J)$), it admits as well functorial factorizations `{\small TrivCof -- Fib}' and `{\small Cof -- TrivFib}' given by the small object argument ({\small SOA}). The latter general technique factors a morphism $\zf:A\to B$ into morphisms \be\label{SOAFac}A\stackrel{i}\longrightarrow C\stackrel{p}{\longrightarrow} B\ee that are obtained as the colimit of a sequence $$A\stackrel{i_{n}}\longrightarrow C_n\stackrel{p_n}{\longrightarrow} B\;,$$ in a way such that $p\in\text{\small RLP}(\cS(J))=\text{\small Fib}$ (resp., $p\in\text{\small RLP}(\cS(I))=\text{\small TrivFib}$). The idea is that, in view of the smallness of the sources in $\cS(J)$ (resp., $\cS(I)$), each commutative square with right down arrow $p:C\to B$ that must admit a lift, factors through a commutative square with right down arrow $p_n:C_n\to B$, and that it therefore suffices to construct $C_{n+1}$ in a way such that `it contains the required lift'. More details can be found in Appendix \ref{SOA}.\medskip

The decompositions (\ref{TrivCofFib}) and (\ref{CofTrivFib}) are $\tt DG\cD A$-specific and different from the general {\small SOA}-factorizations (\ref{SOAFac}). Further, they implement less abstract, in some sense Koszul-Tate type, functorial fibrant and cofibrant resolution functors.\medskip

Before stating the afore-mentioned Theorem \ref{P:c-tf_tc-f}, we sketch the construction of the factorization (\ref{CofTrivFib}). To simplify, we denote algebras of the type $A\0 \cS V_k$ by $R_{V_k}$, or simply $R_k\,$.\medskip

We start from the `small' `{\small Cof -- Fib}' decomposition (\ref{TrivCofFib}) of a $\tt DG\cD A$-morphism $A\stackrel{\zf}{\longrightarrow} B$, i.e., from the factorization $A\stackrel{i}{\longrightarrow}R_U\stackrel{p}{\longrightarrow}B$, see \cite[Section 7.7]{BPP1}. To find a substitute $q$ for $p$, which is a trivial fibration, {\it we mimic an idea used in the construction of the Koszul-Tate resolution: we add generators to improve homological properties}.\medskip

Note first that $H(p)$ is surjective if, for any homology class $[\zb_n]\in H_n(B)$, there is a class $[\zr_n]\in H_n(R_U)$, such that $[p\,\zr_n]=[\zb_n]$. Hence, consider all the homology classes $[\zb_n]$, $n\ge 0,$ of $B$, choose in each class a representative $\dot\zb_n\simeq [\zb_n]$, and add generators $\mbi_{\dot\zb_n}$ to those of $U$. It then suffices to extend the differential $d_1$ (resp., the fibration $p$) defined on $R_U=A\0\cS U$, so that the differential of $\mbi_{\dot\zb_n}$ vanishes (resp., so that the projection of $\mbi_{\dot\zb_n}$ coincides with $\dot\zb_n$) ($\rhd_1$ -- this triangle is just a mark that allows us to retrieve this place later on). To get a {\it functorial} `{\small Cof -- TrivFib}' factorization, we do not add a new generator $\mbi_{\dot\zb_n}$, for each homology class $\dot\zb_n\simeq [\zb_n]\in H_n(B)$, $n\ge 0,$ but we add a new generator $\mbi_{\zb_n}$, for each cycle $\zb_n\in\ker_n d_B$, $n\ge 0\,.$ Let us implement this idea in a rigorous manner. Assign the degree $n$ to $\mbi_{\zb_n}$ and set $$V_0:=U\oplus G_0:= U\oplus \langle \mbi_{\zb_n}: \zb_n\in\ker_n d_B, n\ge 0\ra=$$ \be\label{Newg}\langle s^{-1}\mbi_{b_n}, \mbi_{b_n}, \mbi_{\zb_n}: b_n\in B_n, n>0, \zb_n\in \ker_n d_B, n\ge 0 \ra\;.\ee Set now \be\label{Diffg}\zd_{V_0}(s^{-1}\mbi_{b_n})=d_1(s^{-1}\mbi_{b_n})=0,\;\;\zd_{V_0}\mbi_{b_n}=d_1\mbi_{b_n}={s^{-1}\mbi_{b_n}},\;\;\zd_{V_0}\mbi_{\zb_n}=0\;,\ee thus defining, in view of \cite[Lemma 1]{BPP1}, a differential graded $\cD$-module structure on $V_0$. It follows that $(\cS V_0,\zd_{V_0})\in \tt DG\cD A$ and that \be\label{NewgTot}(R_0,\zd_0):=(A\0\cS V_0,d_A\0\id+\id\0\,\zd_{V_0})\in\tt DG\cD A\;.\ee Similarly, we set \be\label{Morpg}q_{V_0}(s^{-1}\mbi_{b_n})=p(s^{-1}\mbi_{b_n})=\ze(s^{-1}\mbi_{b_n})=d_Bb_n,\;\;q_{V_0}\mbi_{b_n}=p\mbi_{b_n}=\ze\mbi_{b_n}=b_n,\;\;q_{V_0}\mbi_{\zb_n}={\zb}_n\;.\ee We thus obtain \cite[Lemma 2]{BPP1} a morphism $q_{V_0}\in{\tt DG\cD M}(V_0,B)$ -- which uniquely extends to a morphism $q_{V_0}\in{\tt {DG\cD A}}(\cS V_0,B)$. Finally, \be\label{Morp}q_0=\zm_B\circ(\zf\0 q_{V_0})\in {\tt DG\cD A}(R_0,B)\;,\ee where $\zm_B$ denotes the multiplication in $B$. Let us emphasize that $R_U=A\0\cS U$ is a direct summand of $R_0=A\0\cS V_0$, and that $\zd_0$ and $q_0$ just extend the corresponding morphisms on $R_U$: $\zd_0|_{R_U}=d_1$ and $q_0|_{R_U}=p\,$.\medskip

So far we ensured that $H(q_0):H(R_0)\to H(B)$ is surjective; however, it must be injective as well, i.e., for any $\zs_n\in\ker \zd_0$, $n\ge 0,$ such that $H(q_0)[\zs_n]=0$, i.e., such that $q_0\zs_n\in\im d_B$, there should exist $\zs_{n+1}\in R_0$ such that \be\label{Extension}\zs_n=\zd_0\zs_{n+1}\;.\ee We denote by $\cB_0$ the set of $\zd_0$-cycles that are sent to $d_B$-boundaries by $q_0\,$: $${\cB}_0=\{\zs_n\in\ker\zd_0: q_0\zs_n\in\im d_B, n\ge 0\}\;.$$ In principle it now suffices to add, to the generators of $V_0$, generators $\mbi^1_{\zs_n}$ of degree $n+1$, $\zs_n\in{\cB}_0$, and to extend the differential $\zd_0$ on $R_0$ so that the differential of $\mbi^1_{\zs_n}$ coincides with $\zs_n$ ($\rhd_2$). However, it turns out that to obtain a {\it functorial} `{\small Cof -- TrivFib}' decomposition, we must add a new generator $\mbi^1_{\zs_n,\fb_{n+1}}$ of degree $n+1$, for each pair $(\zs_n,\fb_{n+1})$ such that $\zs_n\in\ker\zd_0$ and $q_0\zs_n=d_B\fb_{n+1}\,$: we set \be\label{Critset}{\frak B}_0=\{(\zs_n,\fb_{n+1}):\zs_n\in\ker\zd_0,\fb_{n+1}\in d_B^{-1}\{q_0\zs_n\},n\ge 0\}\ee and \be\label{Newgen}V_1:=V_0\oplus G_1:= V_0\oplus\langle \mbi^1_{\zs_n,\fb_{n+1}}:(\zs_n,\fb_{n+1})\in{\frak B}_0\ra\;.\ee To endow the graded $\cD$-algebra \be\label{VewDGDA}R_1:=A\0\cS V_1\simeq R_0\0\cS G_1\ee with a differential graded $\cD$-algebra structure $\zd_1$, we apply Lemma \ref{LemRSA} (of the present paper), with \be\label{DiffNewGen}\zd_1(\mbi^1_{\zs_n,\fb_{n+1}})=\zs_n\in (R_0)_n\cap \ker\zd_0\;,\ee exactly as suggested by Equation (\ref{Extension}). The differential $\zd_1$ is then given by Equation (\ref{DefRSADiff}) and it extends the differential $\zd_0$ on $R_0$. The extension of the $\tt DG\cD A$-morphism $q_0:R_0\to B$ by a $\tt DG\cD A$-morphism $q_1:R_1\to B$ is built from its definition \be\label{MorpNewGen}q_1(\mbi^1_{\zs_n,\fb_{n+1}})=\fb_{n+1}\in B_{n+1}\cap d_B^{-1}\{q_0\zd_1(\mbi^1_{\zs_n,\fb_{n+1}})\}\ee on the generators and from Equation (\ref{DefRSAMorph}) in Lemma \ref{LemRSA}.\medskip

Eventually, starting from $(R_U,d_1)\in{\tt DG\cD A}$ and $p\in{\tt DG\cD A}(R_U,B)$, we end up -- when trying to make $H(p)$ bijective -- with $(R_{1},\zd_1)\in{\tt DG\cD A}$ and $q_1\in{\tt DG\cD A}(R_{1},B)$ -- so that now $H(q_1):H(R_1)\to H(B)$ must be bijective. Since $(R_1,\zd_1)$ extends $(R_0,\zd_0)$ and $H(q_0):H(R_0)\to H(B)$ is surjective, it is easily checked that this property holds a fortiori for $H(q_1)$. However, when working with $R_1\supset R_0$, the `critical set' ${\cal B}_1\supset {\cal B}_0$ increases, so that we must add new generators $\mbi_{\zs_n}^2$, $\zs_n\in{\cal B}_1\setminus{\cal B}_0$, where $${\cal B}_1=\{\zs_n\in\ker\zd_1:q_1\zs_n\in\im d_B, n\ge 0\}\;.\quad (\rhd_3)$$ To build a {\it functorial$\,$} factorization, we consider not only the `critical set' \be\label{CritSet}{\frak B}_1=\{(\zs_n,\fb_{n+1}):\zs_n\in\ker\zd_1, \fb_{n+1}\in d_B^{-1}\{q_1\zs_n\},n\ge 0\}\;,\ee but also the module of new generators \be\label{NewGen}G_2=\langle\mbi^2_{\zs_n,\fb_{n+1}}:(\zs_n,\fb_{n+1})\in{\frak B}_1\ra\;,\ee indexed, not by ${\frak B}_1\setminus{\frak B}_0$, but by ${\frak B}_1$. Hence an iteration of the procedure (\ref{Critset}) - (\ref{MorpNewGen}) and the definition of a sequence $$(R_0,\zd_0)\rightarrow (R_1,\zd_1)\rightarrow(R_2,\zd_2)\rightarrow\ldots\rightarrow(R_{k-1},\zd_{k-1})\rightarrow (R_{k},\zd_{k})\rightarrow\ldots$$ of canonical inclusions of differential graded $\cD$-algebras $(R_k,\zd_k)$, $R_k=A\0\cS V_k$, $\zd_k|_{R_{k-1}}=\zd_{k-1}$, together with a sequence of ${\tt DG\cD A}$-morphisms $q_k:R_k\to B$, such that $q_k|_{R_{k-1}}=q_{k-1}$. The definitions of the differentials $\zd_k$ and the morphisms $q_k$ are obtained inductively, and are based on Lemma \ref{LemRSA}, as well as on equations of the same type as (\ref{DiffNewGen}) and (\ref{MorpNewGen}).\medskip

The direct limit of this sequence is a differential graded $\cD$-algebra $(R_V,d_2)=(A\0\cS V,d_2)$, together with a morphism $q:A\0\cS V\to B$.

As a set, the colimit of the considered system of canonically included algebras $(R_k,\zd_k)$, is just the union of the sets $R_k$, see Equation (\ref{LimUnion}). We proved above that this set-theoretical inductive limit can be endowed in the standard manner with a differential graded $\cD$-algebra structure and that the resulting algebra {\it is} the direct limit in $\tt DG\cD A$. One thus obtains in particular that $d_2|_{R_k}=\zd_k\,$.

Finally, the morphism $q:R_V\to B$ comes from the universality property of the colimit and it allows to factor the morphisms $q_k:R_k\to B$ through $R_V$. We have: $q|_{R_k}=q_k\,$.\medskip

We will show that this morphism $A\0\cS V\stackrel{q}{\longrightarrow} B$ really leads to a `{\small Cof -- TrivFib'} decomposition $A\stackrel{j}{\longrightarrow} A\0\cS V\stackrel{q}{\longrightarrow} B$ of $A\stackrel{\zf}{\longrightarrow} B$.

\begin{theo}\label{P:c-tf_tc-f}
In $\tt DG\mathcal{D}A$, a functorial `{\small TrivCof -- Fib}' factorization $(i,p)$ and a functorial `{Cof -- TrivFib}' factorization $(j,q)$ of an arbitrary morphism
$$
\zf:(A,d_A)\to (B,d_B)\;,
$$
see Figure \ref{Fact}, can be constructed as follows:\medskip

\begin{figure}[h]
\begin{center}
\begin{tikzpicture}
  \matrix (m) [matrix of math nodes, row sep=3em, column sep=3em]
    {  (A,d_A) & (A\boxtimes\cS U,d_1)  \\
       (A\boxtimes\cS V,d_2) & (B,d_B)  \\ };
 \path[->>]
 (m-1-2) edge  node[right] {$\scriptstyle{p}$} (m-2-2);
 \path[->]
 (m-1-1) edge  node[below] {$\scriptstyle{\zf}$} (m-2-2);
 \path[>->]
 (m-1-1) edge  node[above] {$\sim$} node[below] {$\scriptstyle{i}$} (m-1-2);
  \path[>->]
 (m-1-1) edge  node[left] {$\scriptstyle{j}$} (m-2-1);
  \path[->>]
 (m-2-1) edge  node[above] {$\sim$} node[below] {$\scriptstyle{q}$} (m-2-2);
\end{tikzpicture}
\caption{Functorial factorizations}\label{Fact}
\end{center}
\end{figure}

(1) The module $U$ is the free non-negatively graded $\cD$-module with homogeneous basis $$\bigcup\,\{s^{-1}\mathbb{I}_{b_n},\mathbb{I}_{b_n}\}\;,$$ where the union is over all $b_n\in B_n$ and all $n>0$, and where $\deg({s^{-1}\mathbb{I}_{b_n}})=n-1$ and $\deg(\mathbb{I}_{b_n})=n\,.$ In other words, the module $U$ is a direct sum of copies of the discs $$D^n=\cD \cdot \mbi_{b_n}\oplus \cD \cdot s^{-1}\mbi_{b_n}\;,$$ $n>0$. The differentials $$s^{-1}:D^n\ni \mathbb{I}_{b_n}\to s^{-1}\mathbb{I}_{b_n}\in D^n$$ induce a differential $d_U$ in $U$, which in turn implements a differential $d_S$ in $\cS U$. The differential $d_1$ is then given by $d_1=d_A\0\id+\id\0 d_S\,.$ The trivial cofibration $i:A\to A\0\cS U$ is a minimal split {\small RS$\cD\!$A} defined by $i:\frak{a}\mapsto\frak{a}\0 1_\cO$, and the fibration $p:A\0\cS U\to B$ is defined by $p=\zm_B\circ (\zf\0 \ze)$, where $\zm_B$ is the multiplication of $B$ and where $\ze(\mbi_{b_n})=b_n$ and $\ze(s^{-1}\mbi_{b_n})=d_Bb_n\,$.\medskip

(2) The module $V$ is the free non-negatively graded $\cD$-module with homogeneous basis $$\bigcup\,\{s^{-1}\mathbb{I}_{b_n},\mathbb{I}_{b_n},\mbi_{\zb_n},\mbi^1_{\zs_n,\fb_{n+1}},\mbi^2_{\zs_n,\fb_{n+1}},\ldots,\mbi^k_{\zs_n,\fb_{n+1}},\ldots\}\;,$$ where the union is over all $b_n\in B_n$, $n>0,$ all $\zb_n\in\ker_nd_B$, $n\ge 0$, and all pairs $$(\zs_n,\fb_{n+1}),\; n\ge 0,\;\, \text{in}\;\, {\frak B}_0,{\frak B}_1,\ldots, {\frak B}_k, \ldots,\;$$ respectively. The sequence of sets $${\frak B}_{k-1}=\{(\zs_n,\fb_{n+1}):\zs_n\in\ker\zd_{k-1}, \fb_{n+1}\in d_B^{-1}\{q_{k-1}\zs_n\},n\ge 0\}$$ is defined inductively, together with an increasing sequence of differential graded $\cD$-algebras $(A\0\cS V_k,\zd_k)$ and a sequence of morphisms $q_k:A\0\cS V_k\to B$, by means of formulas of the type (\ref{Critset}) - (\ref{MorpNewGen}) (see also (\ref{Newg}) - (\ref{Morp})). The degrees of the generators of $V$ are \be\label{GenDeg}n-1,\,n,\,n,\,n+1,\,n+1,\ldots, n+1,\ldots\ee The differential graded $\cD$-algebra $(A\0\cS V,d_2)$ is the colimit of the preceding increasing sequence of algebras: \be\label{DefD2}d_2|_{A\0\cS V_k}=\zd_k\;.\ee The trivial fibration $q:A\0\cS V\to B$ is induced by the $q_k$-s via universality of the colimit: \be\label{DefQ}q|_{A\0\cS V_k}=q_k\;.\ee Eventually, the cofibration $j:A\to A\0\cS V$ is a minimal (non-split) {\small RS$\cD\!$A}, which is defined as in (1) as the canonical inclusion; the canonical inclusion $j_k:A\to A\0\cS V_k\,$, $k>0\,$, is also a minimal (non-split) {\small RS$\cD\!$A}, whereas $j_0:A\to A\0\cS V_0$ is a minimal split {\small RS$\cD\!$A}.
\end{theo}

\begin{proof} See Appendix \ref{CofTrivFibProof}. \end{proof}

\begin{rem}\label{NonFuncFact}
{\em \begin{itemize}
\item If we are content with a non-functorial `{\small Cof -- TrivFib}' factorization, we may consider the colimit $A\otimes\cS {\cal V}$ of the sequence $A\otimes\cS {\cal V}_k$ that is obtained by adding only generators (see ($\rhd_1$)) $$\mbi_{\dot\zb_n},\;\, n\ge 0,\;\, \dot\zb_n\simeq[\zb_n]\in H_n(B)\;,$$ and by adding only generators (see ($\rhd_2$) and ($\rhd_3$)) $$\mbi_{\zs_n}^1,\mbi_{\zs_n}^2,\ldots,\;\, n\ge 0,\;\, \zs_n\in {\cal B}_0,{\cal B}_1\setminus{\cal B}_0,\ldots\;$$
\item An explicit description of the functorial fibrant and cofibrant replacement functors, induced by the `{\small TrivCof -- Fib}' and `{\small Cof -- TrivFib}' decompositions of Theorem \ref{P:c-tf_tc-f}, can be found in Appendix \ref{FunctReplFunct}.
\end{itemize}}
\end{rem}

\section{First remarks on Koszul-Tate resolutions}

In this last section, we provide first insight into Koszul-Tate resolutions. Given a polynomial partial differential equation acting on sections of a vector bundle, we obtain, via our preceding constructions, a Koszul-Tate resolution ({\small KTR}) of the corresponding algebra $\cal R$ of on-shell functions. This resolution is a cofibrant replacement of $\cal R$ in the appropriate undercategory of $\tt DG\cD A$.\medskip

In a separate paper \cite{PP}, we give a general and precise definition of Koszul-Tate resolutions. We further show in that work that the classical Tate extension of the Koszul resolution \cite{HT}, the {\small KTR} implemented by a compatibility complex \cite{Ver}, as well as our just mentioned and below detailed model categorical {\small KTR}, are Koszul-Tate resolutions in the sense of this improved definition. Eventually, we investigate the relationships between these three resolutions.\medskip

Hence, the present section should be viewed as an introduction to topics on which we will elaborate in \cite{PP}.

\subsection{Undercategories of model categories}

Given a category $\tt C$ and an object $C\in \tt C$, the {\bf undercategory} or {\bf coslice category} $C\downarrow \tt C$ is the category whose objects are the $\tt C$-morphisms $C\to D$ with source $C$, and whose morphisms between $C\to D_1$ and $C\to D_2$ are the $\tt C$-morphisms $D_1\to D_2$ such that the triangle

\[
\begin{tikzcd}[column sep=1.5em]
 & C \arrow{dr} \arrow[swap]{dl}\\
 D_1 \arrow{rr} && D_2
\end{tikzcd}
\]
commutes. Composition and units are defined in the obvious manner.\medskip

There is a forgetful functor $\op{For}:C\downarrow{\tt C}\to \tt C$ that associates to each $(C\downarrow\tt C)$-object its target and to each $(C\downarrow\tt C)$-morphism its base $D_1\to D_2$. It is customary to write the objects $A$ and morphisms $t$ of the undercategory simply as $\op{For}(A)$ and $\op{For}(t)$ -- whenever no confusion arises (think for instance about smooth vector bundles over a fixed smooth base manifold and corresponding bundle maps). If $\tt C$ is cocomplete, the functor $\op{For}$ has a left adjoint $L_{\amalg}:{\tt C}\to C\downarrow{\tt C}$, which takes a $\tt C$-object $D$ to the morphism $C\to C\coprod D$ and a $\tt C$-morphism $f:D_1\to D_2$ to the commutative triangle

\[
\begin{tikzcd}[column sep=1.5em]
 & C \arrow{dr} \arrow[swap]{dl}\\
 C\coprod D_1 \arrow{rr} && C\coprod D_2
\end{tikzcd}
\]
that is induced via universality by the canonical morphisms $i_{D_2}\circ f:D_1\to C\coprod D_2$ and $i_C:C\to C\coprod D_2$.\medskip

Note also that $\id: C\to C$ is the initial object in $C\downarrow\tt C$, and that, if $\tt C$ has a terminal object $\star$, the unique morphism $C\to \star$ is the terminal object of $C\downarrow\tt C$.\medskip

The next proposition can be found in \cite{Hir2}.

\begin{prop}\label{ModCoSlice} If $C$ is an object of a model category $\tt C$, the coslice category $C\downarrow\tt C$ is also a model category: a $(C\downarrow\tt C)$-morphism $t$ is a cofibration, a fibration, or a weak equivalence, if $\op{For}(t)$ is a cofibration, a fibration, or a weak equivalence in $\tt C$. Moreover, if $\tt C$ is cofibrantly generated with generating cofibrations $I$ and generating trivial cofibrations $J$, the model category $C\downarrow \tt C$ is cofibrantly generated as well, with generating cofibrations $L_{\amalg}I$ and generating trivial cofibrations $L_{\amalg}J$.\end{prop}

When recalling that the coproduct in $\tt DG\cD A$ is the tensor product, we deduce from Theorem 3 in \cite{BPP1} and from Proposition \ref{ModCoSlice} above that:

\begin{cor} For any differential graded $\cD$-algebra $A$, the coslice category $A\downarrow \tt DG\cD A$ carries a cofibrantly generated model structure given by the adjoint pair $L_{\0}:{\tt DG\cD A}\rightleftarrows A\downarrow{\tt DG\cD A}:\op{For}$, in the sense that its distinguished morphism classes are defined by $\op{For}$ and its generating cofibrations and generating trivial cofibrations are given by $L_\0\,$.\end{cor}

Let us conclude by noting that for $A=\cO$ the Quillen adjunction $$L_\0:{\tt DG\cD A}\rightleftarrows\cO\downarrow \tt DG\cD A:\op{For}$$ is obviously an isomorphism of categories.

\subsection{Basics of jet bundle formalism}

The jet bundle formalism allows for a coordinate-free approach to partial differential equations ({\small PDE}-s), i.e., to (not necessarily linear) differential operators ({\small DO}-s) acting between sections of smooth vector bundles (the confinement to vector bundles does not appear in more advanced approaches). To uncover the main ideas, we implicitly consider in this subsection trivialized line bundles $E$ over a 1-dimensional manifold $X$, i.e., we assume that $E\simeq \R\times\R$.\medskip

The key-aspect of the jet bundle approach to {\small PDE}-s is the passage to purely algebraic equations. Consider the order $k$ differential equation ({\small DE}) \be\label{PDE}F(t,\zf(t),d_t\zf,\ldots, d_t^k\zf)=F(t,\zf,\zf',\ldots,\zf^{(k)})|_{j^k\zf}=0\;,\ee where $(t,\zf,\zf',\ldots,\zf^{(k)})$ are coordinates of the $k$-th jet space $J^kE$ and where $j^k\zf$ is the $k$-jet of the section $\zf(t)$. Note that the algebraic equation \be\label{AlgE}F(t,\zf,\zf',\ldots,\zf^{(k)})=0\ee defines a `surface' $\cE^k\subset J^kE$, and that a solution of the considered {\small DE} is nothing but a section $\zf(t)$ whose $k$-jet is located on $\cE^k$.\medskip

A second fundamental feature is that one prefers replacing the original system of {\small PDE}-s by an enlarged system, its infinite prolongation, which also takes into account the consequences of the original one. More precisely, if $\zf(t)$ satisfies the original {\small PDE}, we have also $$d^\ell_t(F(t,\zf(t),d_t\zf,\ldots,d_t^k\zf))=(\p_t+\zf'\p_\zf+\zf''\p_{\zf'}+\ldots)^\ell F(t,\zf,\zf',\ldots,\zf^{(k)})|_{j^\infty\zf}=:$$ \be\label{ProlPDE}D_t^\ell F(t,\zf,\zf',\ldots,\zf^{(k)})|_{j^\infty\zf}=0,\;\forall \ell\in \N\;.\ee Let us stress that the `total derivative' $D_t$ or horizontal lift $D_t$ of $d_t$ is actually an infinite sum. The two systems of {\small PDE}-s, (\ref{PDE}) and (\ref{ProlPDE}), have clearly the same solutions, so we may focus just as well on (\ref{ProlPDE}). The corresponding algebraic system \be\label{ProlAlgE}D_t^\ell F(t,\zf,\zf',\ldots,\zf^{(k)})=0,\;\forall\ell\in\N\;\ee defines a `surface' $\cE^\infty$ in the infinite jet bundle $\zp_\infty:J^\infty E\to X$. A solution of the original system (\ref{PDE}) is now a section $\zf\in\zG(X,E)$ such that $(j^\infty\zf)(X)\subset{\cal E}^\infty$. The `surface' $\cE^\infty$ is often referred to as the `stationary surface' or the `shell'.\medskip

The just described passage from prolonged {\small PDE}-s to prolonged algebraic equations involves the lift of differential operators $d_t^\ell$ acting on $\cO(X)=\zG(X,X\times\R)$ (resp., sending -- more generally -- sections $\zG(X,G)$ of some vector bundle to sections $\zG(X,K)$), to horizontal differential operators $D_t^\ell$ acting on $\cO(J^\infty E)$ (resp., acting from $\zG(J^\infty E,\zp_\infty^*G)$ to $\zG(J^\infty E,\zp_\infty^*K)$). As seen from Equation (\ref{ProlPDE}), this lift is defined by $$(D_t^\ell F)\circ{j^\infty\zf}=d_t^\ell(F\circ j^\infty \zf)\;$$ (note that composites of the type $F\circ j^\infty\zf$, where $F$ is a section of the pullback bundle $\zp_\infty^* G$, are sections of $G$). The interesting observation is that the jet bundle formalism naturally leads to a systematic base change $X\rightsquigarrow J^\infty E$. The remark is fundamental in the sense that both, the classical Koszul-Tate resolution (i.e., the Tate extension of the Koszul resolution of a regular surface) and Verbovetsky's Koszul-Tate resolution (i.e., the resolution induced by the compatibility complex of the linearization of the equation), use the jet formalism to resolve on-shell functions $\cO(\cE^\infty)$, and thus enclose the base change $\bullet\to X$ $\;\rightsquigarrow\;$ $\bullet\to J^\infty E$. This means, dually, that we pass from $\tt DG\cD A$, i.e., from the coslice category $\cO(X)\downarrow \tt DG\cD A$ to the coslice category $\cO(J^\infty E)\downarrow \tt DG\cD A$.

\subsection{Revision of the classical Koszul-Tate resolution}

We first recall the local construction of the {\bf Koszul resolution} of the function algebra $\cO(\zS)$ of a regular surface $\zS\subset\R^n$. Such a surface $\zS$, say of codimension $r$, can locally always be described -- in appropriate coordinates -- by the equations \be\label{E}\zS:x^a=0,\;\forall a\in\{1,\ldots,r\}\;.\ee The Koszul resolution of $\cO(\zS)$ is then the chain complex made of the free Grassmann algebra $$\op{K}=\cO(\R^n)\otimes \cS[\zf^{a*}]$$ on $r$ odd generators $\zf^{a*}$ -- associated to the equations (\ref{E}) -- and of the Koszul differential \be\label{KDiff}\zd_{\op{K}}=x^a\p_{\zf^{a*}}\;.\ee Of course, the claim that this complex is a resolution of $\cO(\zS)$ means that the homology of $(\op{K},\zd_{\op{K}})$ is given by \be\label{Homology}H_0(\op{K})=\cO(\zS)\quad\text{and}\quad H_k(\op{K})=0,\;\forall k>0\;.\ee\medskip

The {\bf Koszul-Tate resolution} of the algebra $\cO(\cE^\infty)$ of on-shell functions is a generalization of the preceding Koszul resolution. In gauge field theory (our main target), $\cE^\infty$ is the stationary surface given by a system \be\label{Eq}\cE^\infty: D_x^\za F_i=0,\;\forall \za,i\;\ee of prolonged algebraized (see (\ref{ProlAlgE})) Euler-Lagrange equations that correspond to some action functional (here $x\in\R^p$ and $\za\in\N^p$). However, there is a difference between the situations (\ref{E}) and (\ref{Eq}): in the latter, there exist gauge symmetries that implement Noether identities and their extensions -- i.e., extensions \be\label{NI}D_x^\zb\; G_{j\za}^i\,D_x^\za F_i=0,\;\forall \zb,j\;\ee of $\cO(J^\infty E)$-linear relations $G_{j\za}^i\,D_x^\za F_i=0$ between the equations $D_x^\za F_i=0$ of $\cE^\infty$ --, which do not have any counterpart in the former. It turns out that, to kill the homology (see (\ref{Homology})), we must introduce additional generators that take into account these relations. More precisely, we do not only associate degree 1 generators $\zf^{\za*}_i$ to the equations (\ref{Eq}), but assign further degree 2 generators $C^{\zb*}_j$ to the relations (\ref{NI}). The Koszul-Tate resolution of $\cO(\cE^\infty)$ is then (under appropriate irreducibility and regularity conditions) the chain complex, whose chains are the elements of the free Grassmann algebra \be\label{KT}\op{KT}=\cO(J^{\infty}E)\otimes \cS[\zf^{\za*}_i,C^{\zb*}_j]\;,\ee and whose differential is defined in analogy with (\ref{KDiff}) by \be\label{KTDiff}\zd_{\op{KT}}=D^\za_xF_i\;\p_{\zf^{\za*}_i}+D_x^\zb\; G^i_{j\za}\,D_x^\za \zf^{*}_i\;\p_{C^{\zb*}_j}\;,\ee where we substituted $\zf^{*}_i$ to $F_i$ (and where total derivatives have to be interpreted in the extended sense that puts the `antifields' $\zf^{*}_i$ and $C^{*}_j$ on an equal footing with the `fields' $\zf^k$ (fiber coordinates of $E$)). The homology of this Koszul-Tate chain complex is actually concentrated in degree 0, where it coincides with $\cO(\cE^\infty)$ (compare with (\ref{Homology})).

\subsection{$\cD$-algebraic version of the Koszul-Tate resolution}

In this subsection, we briefly report on the $\cD$-algebraic approach to `Koszul-Tate' (see \cite{PP} for additional details).

\begin{prop} The functor $$\op{For}:\tt \cD A\to \cO A$$ has a left adjoint $${\cal J}^{\infty}:\tt \cO A\to \cD A\;,$$ i.e., for $B\in\tt \cO A$ and $A\in\tt \cD A$, we have \be\label{JetAlg}\h_{\tt\cD A}({\cal J}^{\infty}(B),A)\simeq \h_{\tt \cO A}(B,\op{For}(A))\;,\ee functorially in $A,B$.\end{prop}

Let now $\zp:E\to X$ be a smooth map of smooth affine algebraic varieties (or a smooth vector bundle). The function algebra $B=\cO(E)$ (in the vector bundle case, we only consider those smooth functions on $E$ that are polynomial along the fibers, i.e., $\cO(E):=\zG(\cS E^*)$) is canonically an $\cO$-algebra, so that the jet algebra ${\cal J}^{\infty}(\cO(E))$ is a $\cD$-algebra. The latter can be thought of as the $\cD$-algebraic counterpart of $\cO(J^\infty E)$. Just as we considered above a scalar {\small PDE} with unknown in $\zG(E)$ as a function $F\in\cO(J^\infty E)$ (see (\ref{AlgE})), an element $P\in{\cal J}^{\infty}(\cO(E))$ can be viewed as a polynomial {\small PDE} acting on sections of $\zp:E\to X$. Finally, the $\cD$-algebraic version of on-shell functions $\cO(\cE^\infty)=\cO(J^\infty E)/(F)$ is the quotient ${\cal R}(E,P):=\cJ^\infty(\cO(E))/(P)$ of the jet $\cD$-algebra by the $\cD$-ideal $(P)$.\medskip

A first candidate for a Koszul-Tate resolution of ${\cal R}:={\cal R}(E,P)\in\tt \cD A$ is of course the cofibrant replacement of ${\cal R}$ in $\tt DG\cD A$ given by the functorial `{\small Cof -- TrivFib}' factorization of Theorem \ref{P:c-tf_tc-f}, when applied to the canonical $\tt DG\cD A$-morphism $\cO\to {\cal R}$. Indeed, this decomposition implements a functorial cofibrant replacement functor $Q$ (see Theorem \ref{FCRF} below) with value $Q({\cal R})=\cS V$ described in Theorem \ref{P:c-tf_tc-f}: $$\cO\rightarrowtail \cS V\stackrel{\sim}{\twoheadrightarrow}{\cal R}\;.$$ Since ${\cal R}$ is concentrated in degree 0 and has 0 differential, it is clear that $H_k(\cS V)$ vanishes, except in degree 0 where it coincides with ${\cal R}$.\medskip

As already mentioned, we propose a general and precise definition of a Koszul-Tate resolution in \cite{PP}. Although such a definition does not seem to exist in the literature, it is commonly accepted that a Koszul-Tate resolution of the quotient of a commutative ring $k$ by an ideal $I$ is an $k$-algebra that resolves $k/I$.\medskip

The natural idea -- to get a ${\cal J}^\infty(\cO(E))$-algebra -- is to replace $\cS V$ by ${\cal J}^\infty(\cO(E))\0 \cS V$, and, more precisely, to consider the `{\small Cof -- TrivFib}' decomposition $${\cal J}^\infty(\cO(E))\rightarrowtail {\cal J}^\infty(\cO(E))\0 \cS V\stackrel{\sim}{\twoheadrightarrow}{\cal J}^\infty(\cO(E))/(P)\;.$$ The {\small DG$\cD$A} \be\label{KTD} {\cal J}^\infty(\cO(E))\0 \cS V\ee {\it is} a {\bf $\cJ^\infty(\cO(E))$-algebra} that {\bf resolves} ${\cal R}=\cJ^\infty(\cO(E))/(P)$, but it is of course {\it not} a cofibrant replacement, since the left algebra is not the initial object $\cO$ in $\tt DG\cD A$ (further, the considered factorization does not canonically induce a cofibrant replacement in $\tt DG\cD A$, since it can be shown that the morphism $\cO\to {\cal J}^\infty(\cO(E))$ is not a cofibration). However, as emphasized above, the Koszul-Tate problem requires a passage from $\tt DG\cD A$ to ${\cal J}^\infty(\cO(E))\downarrow \tt DG\cD A$. It is easily checked that, in the latter undercategory, ${\cal J}^\infty(\cO(E))\0 \cS V$ is a {\bf cofibrant replacement} of ${\cal J}^\infty(\cO(E))/(P)$. To further illuminate the $\cD$-algebraic approach to Koszul-Tate, let us mention why the complex (\ref{KT}) is of the same type as (\ref{KTD}). Just as the variables $\zf^{(k)}$ (see (\ref{PDE})) are algebraizations of the derivatives $d_t^k\zf$ of a section $\zf$ of a vector bundle $E\to X$ (fields), the generators $\zf^{\za*}_i$ and $C^{\zb*}_j$ (see (\ref{Eq}) and (\ref{NI})) symbolize the total derivatives $D_x^\za\zf^*_i$ and $D_x^\zb C^*_j$ of sections $\zf^*$ and $C^*$ of some vector bundles $\zp_\infty^*F_1\to J^\infty E$ and $\zp_\infty^*F_2\to J^\infty E$ (antifields). Hence, the $\zf^{\za*}_i$ and $C^{\zb*}_j$ can be thought of as the horizontal jet bundle coordinates of $\zp_\infty^*F_1$ and $\zp_\infty^*F_2\,$. These coordinates may of course be denoted by other symbols, e.g., by $\p_x^\za\cdot\zf_i^*$ and $\p_x^\zb\cdot C_j^*$, provided we define the $\cD$-action as the action $D_x^\za\zf^*_i$ and $D_x^\zb C^*_j$ by the corresponding horizontal lift, so that we get appropriate interpretations when the $\zf^*_i$-s and the $C^*_j$-s are the components of true sections. This convention allows to write $$\op{KT}=J\0\cS[\p_x^\za\cdot\zf_i^*,\p_x^\zb\cdot C_j^*]=J\0_\cO\cS_\cO(\oplus_i\,\cD\cdot\zf_i^*\;\oplus\; \oplus_j\,\cD\cdot C_j^*)\;,$$ where $J={\cal J}^\infty(\cO(E))\,,$ so that the space (\ref{KT}) is really of the type (\ref{KTD}). Let us emphasize that (\ref{KT}) and (\ref{KTD}), although of the same type, are of course not equal (for instance, the classical Koszul-Tate resolution is far from being functorial). For further details, see \cite{PP}.

\section{Appendix}

\subsection{Small object argument}\label{SOA}

The `{\small TrivCof -- Fib}' and `{\small Cof -- TrivFib}' factorizations of a cofibrantly generated model category can be constructed in a functorial way. The constructions use an {\it argument} that is based on the fact that the sources of the morphisms in $I$ and $J$ are {\it small objects} -- the so-called small object argument ({\small SOA}), which goes back to Quillen. Although this argument is described elsewhere in the literature, we provide a compact description that allows to compare our $\tt DG\cD A$-specific factorizations with the general {\small SOA}-factorizations.\medskip

In the following, $\tt C$ is just a category with all small colimits, $W$ is a set of $\tt C$-morphisms, whose sources are sequentially small, see \cite[Sections 8.5 and 8.6]{BPP1}. Our goal is to decompose any $\tt C$-morphism $f:A\to B$ as $A\stackrel{j}{\longrightarrow} C\stackrel{q}{\longrightarrow} B$, where $q\in\text{\small RLP}(W)$ (we will not show that this factorization leads to functorial `{\small TrivCof -- Fib}' and `{\small Cof -- TrivFib}' factorizations).\medskip

The intermediate object $C$ and the morphism $q$ will be constructed as the colimit of an $\zw$-sequence:
\be\label{Sequence}\begin{array}{ccccccccccccccccc}&A&\stackrel{j_0}{\longrightarrow}&&C_0&\stackrel{j_1}{\longrightarrow}&\ldots&\stackrel{j_n}{\longrightarrow}&& C_n&\stackrel{j_{n+1}}{\longrightarrow}&&C_{n+1}&\stackrel{j_{n+2}}{\longrightarrow}&\ldots&& C\\ f&\downarrow&&q_0&\downarrow&&&&q_n&\downarrow&&q_{n+1}&\downarrow&&&q&\downarrow\\&B&=&&B&=&&=&&B&=&&B&=&\ldots&&B\end{array}\ee
The construction starts with the first commutative square in the preceding diagram, where $(C_0,j_0,q_0)=(A,\id,f)\,.$ Assume now that the construction is done up to the commutative square $(C_n, j_n,q_n)$ inclusively, set as usual $j_{n0}=j_n\circ\ldots\circ j_0$, and memorize that $q_n\circ j_{n0}=f$.\medskip

Before constructing the commutative square $(C_{n+1},j_{n+1},q_{n+1})$, recall that we wish to get $q\in\text{\small RLP}(W)$, i.e., that any commutative square of $\tt C$-morphisms
\begin{center}
\begin{tikzpicture}
  \matrix (m) [matrix of math nodes, row sep=3em, column sep=3em]
    {  U & C  \\
       V & B  \\ };
 \path[->,dashed]
 (m-2-1) edge  node[right] {$\scriptstyle{\ell}$} (m-1-2);
 \path[->]
 (m-1-1) edge node[above] {$\scriptstyle{\zf}$} (m-1-2);
 \path[->]
 (m-1-1) edge  node[left] {$\scriptstyle{w}$} (m-2-1);
 \path[->]
 (m-1-2) edge  node[right] {$\scriptstyle{q}$} (m-2-2);
 \path[->]
 (m-2-1) edge  node[below] {$\scriptstyle{\psi}$}(m-2-2);
\end{tikzpicture}
\end{center}

\noindent with $w\in W$ must admit a lift $\ell$. In other words, we have to build the colimit $C$ in such a way that this lift does exist. Note now that, since $U$ is sequentially small, the morphism $\zf:U\to C=\colim_n C_n$ will factor through some stage of the colimit, i.e., that $\zf$ will be the composite of a morphism $\zf_n:U\to C_n$ and the transfinite composite $j_{\infty n}=\ldots\circ j_{n+2}\circ j_{n+1}:C_n\to C$:
\be\label{Sequence2}\begin{array}{cccccccccccc}&U&\stackrel{\zf_n}{\longrightarrow}&& C_n&\stackrel{j_{n+1}}{\longrightarrow}&&C_{n+1}&\stackrel{j_{n+2}}{\longrightarrow}&\ldots&& C\\ w&\downarrow&&q_n&\downarrow&&q_{n+1}&\downarrow&&&q&\downarrow\\&V&\stackrel{\psi}{\longrightarrow}&&B&=&&B&=&\ldots&&B\end{array}\ee
Therefore, we define the commutative square $(C_{n+1},j_{n+1},q_{n+1})$ as follows. Let $S$ be the set of all commutative squares

\begin{center}
\begin{tikzpicture}
  \matrix (m) [matrix of math nodes, row sep=3em, column sep=3em]
    {  U & C_n  \\
       V & B  \\ };
 \path[->]
 (m-1-1) edge (m-1-2);
 \path[->]
 (m-1-1) edge  node[left] {$\scriptstyle{w}$} (m-2-1);
 \path[->]
 (m-1-2) edge  node[right] {$\scriptstyle{q_n}$} (m-2-2);
 \path[->]
 (m-2-1) edge  (m-2-2);
\end{tikzpicture}
\end{center}

\noindent with $w\in W$. Due to universality of a coproduct, we then get a commutative square

\begin{center}
\begin{tikzpicture}
  \matrix (m) [matrix of math nodes, row sep=3em, column sep=3em]
    {  \coprod_SU & C_n  \\
       \coprod_SV & B  \\ };
 \path[->]
 (m-1-1) edge (m-1-2);
 \path[->]
 (m-1-1) edge  node[left] {$\scriptstyle{\coprod_Sw}$} (m-2-1);
 \path[->]
 (m-1-2) edge  node[right] {$\scriptstyle{q_n}$} (m-2-2);
 \path[->]
 (m-2-1) edge  (m-2-2);
\end{tikzpicture}
\end{center}

\noindent We now define $C_{n+1}$ to be the pushout of the upper and left arrows of the latter square, and obtain morphisms $j_{n+1}:C_n\to C_{n+1}$ and $\ell_{n+1}:\coprod_SV\to C_{n+1}$, and, in view of universality of a pushout, a morphism $q_{n+1}:C_{n+1}\to B$ such that, in particular, $q_{n+1}\circ j_{n+1}=q_n$, with the result that $q_{n+1}\circ j_{n+1,0}=q_n\circ j_{n0}=f$.\medskip

This leads to the commutative diagram (\ref{Sequence}). We take its colimit, i.e., we set $C=\op{colim}_nC_n$ and get $j_{\infty n}:C_n\to C$ and $j=j_{\infty n}\circ j_{n0}:A\to C$, as well as, from the universality of a colimit, $q:C\to B$ such that $q\circ j_{\infty n}=q_n$. Hence, the factorization $$f=q_n\circ j_{n0}=q\circ j_{\infty n}\circ j_{n0}=q\circ j\;.$$

To show that $q\in\text{\small RLP}(W)$, consider a commutative square $q\circ\zf=\psi\circ w$ as above. Since $\zf=j_{\infty n}\circ\zf_n$ and $q\circ j_{\infty n}=q_n$, it induces a commutative square $q_n\circ \zf_n=\psi\circ w$ as in Figure (\ref{Sequence2}), which is used to build the pushout $C_{n+1}$. Hence, a morphism $\ell_{n+1}:V\to C_{n+1}$ and a morphism $\ell=j_{\infty,n+1}\circ \ell_{n+1}:V\to C$. The latter is quite easily seen to be the searched lift.

\subsection{Proof of Theorem \ref{P:c-tf_tc-f}}\label{CofTrivFibProof}

The proof of functoriality of the decompositions will be given in Appendix \ref{FunctReplFunct}. Thus, only Part (2) requires immediate explanations. We use again the above-introduced notation $R_k=A\0\cS V_k$; we also set $R=A\0\cS V$. The multiplication in $R_k$ (resp., in $R$) will be denoted by $\diamond_k$ (resp., $\diamond$). \medskip

To show that $j$ is a minimal {\small RS$\cD\!$A}, we have to check that $A$ is a differential graded $\cD$-subalgebra of $R$, that the basis of $V$ is indexed by a well-ordered set, that $d_2$ is lowering, and that the minimality condition (\ref{minimal}) is satisfied.

The main idea to keep in mind is that $R=\bigcup_kR_k\,$ -- so that any element of $R$ belongs to some $R_k$ in the increasing sequence $R_0\subset R_1\subset \ldots\,$ -- and that the {\small DG$\cD$A} structure on $R$ is defined in the standard manner. For instance, the product of ${\frak a}\0 X,{\frak b}\0 Y\in R\cap R_k$ is defined by $$({\frak a}\0 X)\diamond ({\frak b}\0 Y)=({\frak a}\0 X)\diamond_k ({\frak b}\0 Y)=(-1)^{\tilde X\tilde{\frak b}}({\frak a}\ast{\frak b})\0 (X\odot Y)\;,$$ where `tilde' (resp., $\ast$) denotes as usual the degree (resp., the multiplication in $A$). It follows that $\diamond$ restricts on $A$ to $\ast\,$. Similarly, $d_2|_A=\zd_0|_A=d_A$, in view of (\ref{DefD2}) and (\ref{NewgTot}). Finally, we see that $A$ satisfies actually the mentioned subalgebra condition.

We now order the basis of $V$. First, we well-order, for any fixed generator degree $m\in\N$ (see (\ref{GenDeg})), the sets \be\label{OrdSetsGen}\{s^{-1}\mbi_{b_{m+1}}\},\, \{\mbi_{b_m}\},\, \{\mbi_{\zb_m}\},\, \{\mbi^1_{\zs_{m-1},\fb_{m}}\},\, \{\mbi^2_{\zs_{m-1},\fb_{m}}\},\, \ldots\ee of degree $m$ generators of a given type (for $m=0$, only the sets $\{s^{-1}\mbi_{b_1}\}$ and $\{\mbi_{\zb_0}\}$ are non-empty). We totally order the set of all degree $m$ generators by totally ordering its partition (\ref{OrdSetsGen}): $$\{s^{-1}\mbi_{b_{m+1}}\}<\{\mbi_{b_m}\}<\{\mbi_{\zb_m}\}<\{\mbi^1_{\zs_{m-1},\fb_{m}}\}<\{\mbi^2_{\zs_{m-1},\fb_{m}}\}<\ldots\;$$ A total order on the set of all generators (of all degrees) is now obtained by declaring that any generator of degree $m$ is smaller than any generator of degree $m+1$. This total order is a well-ordering, since no infinite descending sequence exists in the set of all generators. Observe that our well-order respects the degree (in the sense of (\ref{minimal})).

Finally, the differential $d_2$ sends the first and third types of generators (see (\ref{OrdSetsGen})) to 0 and it maps the second type to the first. Hence, so far $d_2$ is lowering. Further, we have $$d_2(\mbi^k_{\zs_{m-1},\fb_{m}})=\zs_{m-1}\in (R_{k-1})_{m-1}\;,$$ where $m-1$ refers to the term of degree $m-1$ in $R_{k-1}$. Since this term is generated by the generators $$\{s^{-1}\mbi_{b_{\ell+1}}\}, \{\mbi_{b_\ell}\},\{\mbi_{\zb_\ell}\},\{\mbi^1_{\zs_{\ell-1},\fb_{\ell}}\},\ldots, \{\mbi^{k-1}_{\zs_{\ell-1},\fb_{\ell}}\}\,,$$ where $\ell < m$, the differential $d_2$ is definitely lowering.\medskip

It remains to verify that the described construction yields a morphism $q:A\0\cS V\to B$ that is actually a trivial fibration.

Since fibrations are exactly the morphisms that are surjective in all positive degrees, and since $q|R_U=q_0|R_U=p$ is degree-wise surjective, it is clear that $q$ is a fibration. As for triviality, let $[\zb_n]\in H(B,d_B)$, $n\ge 0\,$. Since $\mbi_{\zb_n}\in \ker \zd_0\subset \ker d_2$, the homology class $[\mbi_{\zb_n}]\in H(R,d_2)$ makes sense; moreover, $$H(q)[\mbi_{\zb_n}]=[q\mbi_{\zb_n}]=[q_0\mbi_{\zb_n}]=[\zb_n]\;,$$ so that $H(q)$ is surjective. Eventually, let $[\zs_n]\in H(R,d_2)$ and assume that $H(q)[\zs_n]=0$, i.e., that $q\zs_n\in\im d_B$. Since there is a lowest $k\in\N$ such that $\zs_n\in R_k$, we have $\zs_n\in\ker\zd_k$ and $q_k\zs_n=d_B\fb_{n+1}$, for some $\fb_{n+1}\in B_{n+1}$. Hence, a pair $(\zs_n,\fb_{n+1})\in{\frak B}_k$ and a generator $\mbi^{k+1}_{\zs_n,\fb_{n+1}}\in R_{k+1}\subset R$. Since $$\zs_n=\zd_{k+1}\mbi^{k+1}_{\zs_n,\fb_{n+1}}=d_2\mbi^{k+1}_{\zs_n,\fb_{n+1}}\;,$$ we obtain that $[\zs_n]=0$ and that $H(q)$ is injective.

\subsection{Explicit fibrant and cofibrant functorial replacement functors}\label{FunctReplFunct}

(1) We proved already \cite[Theorem 4]{BPP1} that the factorization $(i,p)=(i(\zf),p(\zf))$ of the $\tt DG\cD A$-morphisms $\zf$, described in Theorem \ref{P:c-tf_tc-f}, is functorial, i.e., that, for any commutative $\tt DG\cD A$-square \be\label{InitSq1}
\xymatrix{
A\ar[d]^{u} \ar[r]^{\zf}&B\ar[d]^{v\;\;\;,}\\
A'\ar[r]^{\zf'}&B'\\
}
\ee
there is a commutative $\tt DG\cD A$-diagram
\be\label{CompMor0}
\xymatrix{A\;\; \ar[d]_{u} \ar^{\sim}_{i(\zf)}  @{>->} [r] & A\0\cS U \ar[d]^{w}\;\;\ar @{->>} [r]_{p(\zf)} & B \ar[d]^{v\;\;\;.}\\
A'\;\; \ar @{>->} [r]^{\sim}_{i(\zf')} & A'\0\cS U'\;\; \ar @{->>} [r]_{p(\zf')} & B'\\
}
\ee
The $\tt DG\cD A$-morphism $w$ is given by $w=u\0 \tilde{v}$, where $\tilde{v}$ is the $\tt DG\cD A$-morphism $\tilde{v}:\cS U\to \cS U'$ defined by $$\tilde{v}(s^{-1}\mbi_{b_n})=s^{-1}\mbi_{v(b_n)}\in\cS U'\;\,\text{and}\;\,\tilde{v}(\mbi_{b_n})=\mbi_{v(b_n)}\in\cS U'\;.$$

\begin{prop} In $\tt DG\cD A$, the functorial fibrant replacement functor $R$, which is induced by the functorial `{\small TrivCof -- Fib}' factorization $(i,p)$ of Theorem \ref{P:c-tf_tc-f}, is the identity functor: $R=\id$. In particular, all objects are fibrant. \end{prop}

\begin{proof} When applying the decomposition $(i,p)$ to the commutative square \be\label{InitSq1}
\xymatrix{
A\ar[d]^{u} \ar[r]^{z_A}&\{0\}\ar[d]^{0\;\;\;,}\\
A'\ar[r]^{z_{A'}}&\{0\}\\
}
\ee
we get \be\label{CompMor1}
\xymatrix{A\;\; \ar[d]_{u} \ar^{\sim}_{i_A}  @{>->} [r] & A\0\cO \ar[d]^{u\0 \id}\;\;\ar @{->>} [r]_{z_{A\0\cO}} & \{0\} \ar[d]^{0\;\;\;.}\\
A'\;\; \ar @{>->} [r]^{\sim}_{i_{A'}} & A'\0\cO\;\; \ar @{->>} [r]_{z_{A'\0\cO}} & \{0\}\\
}
\ee\smallskip

\noindent It follows that the functorial fibrant replacement functor $R$ maps $A$ (resp., $u$) to $R(A)=A\0_\cO \cO\simeq A$ (resp., $R(u)=u\0\id\simeq u\,$).
\end{proof}

(2) To finish the proof of Theorem \ref{P:c-tf_tc-f}, we still have to show that the factorization $(j,q)$ is functorial, i.e., that for any commutative $\tt DG\cD A$-square \be\label{InitSq1}
\xymatrix{
A\ar[d]^{u} \ar[r]^{\zf}&B\ar[d]^{v\;\;\;,}\\
A'\ar[r]^{\zf'}&B'\\
}
\ee
there is a commutative $\tt DG\cD A$-diagram
\be\label{CompMor2}
\xymatrix{A\;\; \ar[d]_{u} \ar_{j:=j(\zf)}  @{>->} [r] & A\0\cS V \ar[d]^{\zw}\;\;\ar @{->>} [r]^{\sim}_{q:=q(\zf)} & B \ar[d]^{v\;\;\;.}\\
A'\;\; \ar @{>->} [r]_{j':=j(\zf')} & A'\0\cS V'\;\; \ar @{->>} [r]^{\sim}_{q':=q(\zf')} & B'\\
}
\ee\smallskip

Let us stress that the following proof fails, if we use the non-functorial factorization mentioned in Remark \ref{NonFuncFact} (the critical spots are marked by $\triangleleft\,$).

Just as we constructed in Section \ref{Factorizations}, the {\small RS$\cD$A} $R=A\0\cS V$ (resp., $R'=A'\0\cS V'$) as the colimit of a sequence $R_k=A\0\cS V_k$ (resp., $R'_k=A'\0\cS V'_k$), we will build $\zw\in{\tt DG\cD A}(R,R')$ as the colimit of a sequence \be\label{w_k}\zw_k\in{\tt DG\cD A}(R_k,R'_k)\;.\ee Recall moreover that $q$ is the colimit of a sequence $q_k\in{\tt DG\cD A}(R_k,B)$, and that $j$ is nothing but $j_k\in{\tt DG\cD A}(A,R_k)$ viewed as valued in the supalgebra $R$ -- and similarly for $q',q'_k,j',j'_k$. Since we look for a morphism $\zw$ that makes the left and right squares of the diagram (\ref{CompMor2}) commutative, we will construct $\zw_k$ so that \be\label{ComSqu}\zw_k\,j_k=j'_k\,u\;\,\text{and}\;\; v\,q_k=q'_k\,\zw_k\;.\ee

Since the {\small RS$\cD$A} $A\to R_0=A\0\cS V_0$ is split, we define $$\zw_0\in{\tt DG\cD A}(A\0\cS V_0, R'_0)$$ as \be\label{zw_0}\zw_0=j'_0\,u\diamond_0 w_0\;,\ee where we denoted the multiplication in $R'_0\,$ by the same symbol $\diamond_0$ as the multiplication in $R_0$, where $j'_0\,u\in{\tt DG\cD A}(A,R'_0)$, and where $w_0\in{\tt DG\cD A}(\cS V_0, R'_0)$. As the differential $\zd_{V_0}$, see Section \ref{Factorizations}, has been obtained via \cite[Lemma 1]{BPP1}, the morphism $w_0$ can be built as described in \cite[Lemma 2]{BPP1}: we set \be\label{w_0}w_0(s^{-1}\mbi_{b_n})=s^{-1}\mbi_{v(b_n)}\in V'_0\,,\;w_0(\mbi_{b_n})=\mbi_{v(b_n)}\in V'_0\,,\;\, \text{and}\;\; w_0(\mbi_{\zb_n})=\mbi_{v(\zb_n)}\in V'_0\;,\ee and easily check that $w_0\,\zd_{V_0}=\zd'_0\, w_0$ on the generators. The first commutation condition (\ref{ComSqu}) is obviously satisfied. As for the verification of the second condition, let $t={\frak a}\0 x_1\odot\ldots\odot x_\ell\in A\0\cS V_0$ and remember (see (\ref{Morp})) that $q_0=\zf\star q_{V_0}$ and $q'_0=\zf'\star q_{V'_0}\,$, where we denoted again the multiplications in $B$ and $B'$ by the same symbol $\star$. Then $$vq_0(t)=v\zf({\frak a})\star vq_{V_0}(x_1)\star\ldots\star vq_{V_0}(x_\ell)$$ and $$q'_0\zw_0(t)=q'_0j'_0u({\frak a})\star q'_0w_0(x_1)\star\ldots\star q'_0w_0(x_\ell)=\zf'u({\frak a})\star q'_0w_0(x_1)\star\ldots\star q'_0w_0(x_\ell)\;.$$ It thus suffices to show that $v\,q_{V_0}=q'_0\,w_0$ on the generators $s^{-1}\mbi_{b_n}, \mbi_{b_n}, \mbi_{\zb_n}$ of $V_0$, what follows from Equations (\ref{Morpg}) and (\ref{w_0}) ($\triangleleft_1$).

Assume now that the $\zw_\ell$ have been constructed according to the requirements (\ref{w_k}) and (\ref{ComSqu}), for all $\ell\in\{0,\dots,k-1\}$, and build their extension $$\zw_k\in{\tt DG\cD A}(R_k,R'_k)$$ as follows. Since $\zw_{k-1}$, viewed as valued in $R'_k$, is a morphism $\zw_{k-1}\in{\tt DG\cD A}(R_{k-1},R'_k)$ and since the differential $\zd_k$ of $R_k\simeq R_{k-1}\0 \cS G_k$, where $G_k$ is the free $\cD$-module $$G_k=\langle \mbi^k_{\zs_n,\fb_{n+1}}:(\zs_n,\fb_{n+1})\in{\frak B}_{k-1}\ra\;,$$ has been defined by means of Lemma \ref{LemRSA}, the morphism $\zw_k$ is, in view of the same lemma, completely defined by degree $n+1$ values $$\zw_k(\mbi^k_{\zs_n,\fb_{n+1}})\in \zd'^{-1}_k(\zw_{k-1}\zd_k(\mbi^k_{\zs_n,\fb_{n+1}}))\;.$$ As the last condition reads $$\zd'_k\,\zw_k(\mbi^k_{\zs_n,\fb_{n+1}})=\zw_{k-1}(\zs_n)\;,$$ it is natural to set \be\label{w_kDef}\zw_k(\mbi^k_{\zs_n,\fb_{n+1}})=\mbi^k_{\zw_{k-1}(\zs_n),v(\fb_{n+1})}\;,\ee provided we have $$(\zw_{k-1}(\zs_n),v(\fb_{n+1}))\in{\frak B}'_{k-1}\quad (\triangleleft_2)\;.$$ This requirement means that $\zd'_{k-1}\zw_{k-1}(\zs_n)=0$ and that $q'_{k-1}\zw_{k-1}(\zs_n)=d_{B'}\,v(\fb_{n+1})$. To see that both conditions hold, it suffices to remember that $(\zs_n,\fb_{n+1})\in{\frak B}_{k-1}$, that $\zw_{k-1}$ commutes with the differentials, and that it satisfies the second equation (\ref{ComSqu}). Hence the searched morphism $\zw_k\in{\tt DG\cD A}(R_k,R'_k)$, such that $\zw_k|_{R_{k-1}}=\zw_{k-1}$ (where the {\small RHS} is viewed as valued in $R'_k$). To finish the construction of $\zw_k$, we must still verify that $\zw_k$ complies with (\ref{ComSqu}). The first commutation relation is clearly satisfied. For the second, we consider $$r_k=r_{k-1}\0 g_1\odot\ldots\odot g_\ell\in R_{k-1}\0\cS G_k\;$$ and proceed as above: recalling that $\zw_k$ and $q_k$ have been defined via Equation (\ref{DefRSAMorph}) in Lemma \ref{LemRSA}, that $q'_k$ and $v$ are algebra morphisms, and that $\zw_{k-1}$ satisfies (\ref{ComSqu}), we see that it suffices to check that $q'_k\,\zw_k=v\,q_k$ on the generators $\mbi^k_{\zs_n,\fb_{n+1}}$ -- what follows immediately from the definitions ($\triangleleft_3$).

Remember now that $((R,d_2),i_r)$ is the direct limit of the direct system $((R_k,\zd_k),\iota_{sr})$, i.e., that
 \begin{equation} \begin{tikzpicture}
 \matrix (m) [matrix of math nodes, row sep=3em, column sep=3em]
   {  R_0  & \cdots & R_k & \cdots  \\
       & R \\ };
 \path[>->]
 (m-1-1) edge node[right] {\small{$i_0$}} (m-2-2)
 (m-1-3) edge node[right] {\small{$i_k$}} (m-2-2)
 (m-1-3) edge node[above] {\small{$\iota_{k+1,k}$}} (m-1-4)
 (m-1-1) edge node[above] {\small{$\iota_{10}$}}(m-1-2)
 (m-1-2) edge node[above] {\small{$\iota_{k,k-1}$}}(m-1-3);
\end{tikzpicture}\;
\end{equation}
where all arrows are canonical inclusions, and that the same holds for $((R',d'_2),i'_r)$ and $((R'_k,\zd'_k),\iota'_{sr})$. Since the just defined morphisms $\zw_k$ provide morphisms $i'_k\,\zw_k\in{\tt DG\cD A}(R_k,R')$ (such that the required commutations hold -- as $\zw_k|_{R_0}=\zw_0$), it follows from universality that there is a unique morphism $\zw\in{\tt DG\cD A}(R,R')$, such that $\zw\,i_k=i'_k\,\zw_k\,$, i.e., such that \be\label{w}\zw|_{R_k}=\zw_k\;.\ee When using the last result, one easily concludes that $\zw\,j=j'\,u$ and $v\,q=q'\,\zw\,$.\medskip

This completes the proof of Theorem \ref{P:c-tf_tc-f}.\medskip

\begin{rem} The preceding proof of functoriality fails for the factorization of Remark \ref{NonFuncFact}. The latter adds only one new generator $\mbi_{\dot\zb_n}$ for each homology class $\dot\zb_n\simeq[\zb_n]$, and it adds only one new generator $\mbi^k_{\zs_n}$ for each $\zs_n\in {\cal B}_{k-1}\setminus {\cal B}_{k-2}\,$, where $${\cal B}_r=\{\zs_n\in\ker\zd_r:q_r\zs_n\in \im d_B,n\ge 0\}\;.$$ In $(\,$$\triangleleft_1$$\,)$, we then get that $v\,q_{V_0}(\mbi_{\dot\zb_n})$ and $q'_0\,w_0(\mbi_{\dot\zb_n})$ are homologous, but not necessarily equal. In $(\,$$\triangleleft_2$$\,)$, although $\zs_n\in\cB_{k-1}\setminus\cB_{k-2}$, its image $\zw_{k-1}(\zs_n)\in\cB'_{k-1}$ may also belong to $\cB'_{k-2}\,$. Eventually, in $(\,$$\triangleleft_3$$\,)$, we find that  $vq_k(\mbi^k_{\zs_n})$ and $q'_k\zw_k(\mbi^k_{\zs_n})$ differ by a cycle, but do not necessarily coincide.\end{rem}

The next result describes cofibrant replacements.

\begin{theo}\label{FCRF} In $\tt DG\cD A$, the functorial cofibrant replacement functor $Q$, which is induced by the functorial `{\small Cof -- TrivFib}' factorization $(j,q)$ described in Theorem \ref{P:c-tf_tc-f}, is defined on objects $B\in\tt DG\cD A$ by $Q(B)=\cS V_B$, see Theorem \ref{P:c-tf_tc-f} and set $A=\cO$, and on morphisms $v\in{\tt DG\cD A}(B,B')$ by $Q(v)=\zw$, see Equations (\ref{w}), (\ref{w_kDef}), and (\ref{w_0}), and set $\zw_0=w_0$. Moreover, the differential graded $\cD$-algebra $\cS{\cal V}_B$, see Proposition \ref{NonFuncFact} and set $A=\cO$, is a cofibrant replacement of $B$. \end{theo}

\begin{proof} Since the initial object in $\tt DG\cD A$ is $(\cO,0)$, it suffices to apply the afore-detailed construction of the commutative diagram (\ref{CompMor2}) to the commutative square \be\xymatrix{
\cO\ar[d]^{\id} \ar[r]^{I_B}&B\ar[d]^{v\;\;\;,}\\
\cO\ar[r]^{I_{B'}}&B'\\
}\ee
where $I_B$ is defined by $I_B(1_\cO)=1_B$, and similarly for $I_{B'}\,$.\end{proof}


\begin{thebibliography}{Dillo 83}

\bibitem[Bar10]{Bar} G. Barnich, {\em Global and gauge symmetries in classical field theories}, Series of lectures, Seminar `Algebraic Topology, Geometry and Physics', University of Luxembourg, {\tt homepages.ulb.ac.be/$\sim$gbarnich}.

\bibitem[BD04]{BD} A. Beilinson and V. Drinfeld, {\em Chiral algebras}, American Mathematical Society Colloquium Publications, {\bf 51}, American Mathematical Society, Providence, RI, 2004.

\bibitem[BPP15a]{BPP1} G. di Brino, D. Pi\v{s}talo, and N. Poncin, {\em Model structure on differential graded commutative algebras over the ring of differential operators}, to appear in ArXiv.

\bibitem[BPP16]{BPP3} G. di Brino, D. Pi\v{s}talo, and N. Poncin, {\em Derived algebraic $\cD$-stacks}, to appear in ArXiv.

\bibitem[Cos11]{Cos} K. Costello, {\em Renormalization and Effective Field Theory}, Mathematical Surveys and Monographs Volume, {\bf 170}, American Mathematical Society, 2011.

\bibitem[DS96]{DS} W. G. Dwyer, J. Spalinski, \emph{Homotopy theories anad model categories}, Springer, 1996.

\bibitem[FHT01]{FHT} Y. F\'elix, S. Halperin, J.-C. Thomas, {\em Rational Homotopy Theory}, Graduate Texts in Mathematics, {\bf 205}, Springer, 2001.

\bibitem[GM96]{GM} S. I. Gelfand, Y. I. Manin \emph{Methods of Homological Algebra}, Springer, 1996.

\bibitem[GS06]{GS} P. G. Goerss, K. Schemmerhorn, \emph{Model Categories and Simplicial Methods}, arXiv:math/0609537.

\bibitem[Har97]{Har} R. Hartshorne, {\em Algebraic Geometry}, Graduate Texts in Mathematics {\bf 52}, Springer, 1997.

\bibitem[HT92]{HT} M. Henneaux and C. Teitelboim, {\em Quantization of Gauge Systems}, Princeton University Press, 1992.

\bibitem[Hes00]{Hes} K. Hess, \emph{Rational Homotopy Theory: A Brief introduction}, 2000.

\bibitem[Hir00]{Hir} P. Hirschhorn, {\em Model Categories and Their Localizations}, Mathematical Surveys and Monographs {\bf 99}, American Mathematical Society, 2000.

\bibitem[Hir05]{Hir2} P. Hirschhorn, {\em Overcategories and undercategories of model categories}, {\tt http://www-math.mit.edu/~psh/undercat.pdf}, 2005.

\bibitem[HTT08]{HTT} R. Hotta, K. Takeuchi, and T. Tanisaki, {\em $\cD$-Modules, Perverse Sheaves, and Representation Theory}, Progress in Mathematics, {\bf 236}, Birkh\"auser, 2008.

\bibitem[Hov07]{Hov} M. Hovey, \emph{Model Categories}, American Mathematical Society, 2007.

\bibitem[KS90]{KS} M. Kashiwara, P. Schapira, {\em Sheaves on Manifolds}, Springer Science and Business Media, 1990.

\bibitem[Mum99]{Mum} D. Mumford, {\em The Red Book of Varieties and Schemes}, Lecture Notes in Mathematics {\bf 1358}, Springer, 1999.

\bibitem[Nes03]{Nes} J. Nestruev, {\em Smooth manifolds and observables}, Graduate texts in mathematics {\bf 220}, Springer, 2003.

\bibitem[Pau11]{Paugam} F. Paugam, {\em Histories and observables in covariant field theory}, J. Geo. Phys., {\bf 61 (9)}, 2011, 1675-1702.

\bibitem[PP16]{PP} D. Pi\v{s}talo, N. Poncin, {On four Koszul-Tate resolutions}, to appear in ArXiv, 2016.

\bibitem[Sch12]{Scha} P. Schapira, {\em D-modules}, lecture notes,\\ {\tt http://www.math.jussieu.fr/$\sim$schapira/lectnotes/Dmod.pdf.}

\bibitem[Sch94]{Schn} J.-P. Schneiders {\em An introduction to $\cD$-modules}, Bulletin de la Soci\'et\'e Royale des Sciences de Li\`ege, 1994.

\bibitem[Ser55]{Serre}  J.-P. Serre, {\it Faisceaux Alg\'ebriques Coh\'erents}, Ann. Math., 2nd Ser., {\bf 61 (2)}, 1955, 197-278.

\bibitem[Swa62]{Swan} R. G. Swan, {\em Vector Bundles and Projective Modules}, Transactions of the American Mathematical Society {\bf 105 (2)}, 1962, 264-277.

\bibitem[TV04]{TV04} B. To\"en, G. Vezzosi, {\em From HAG to DAG: derived moduli stacks, Axiomatic, enriched and motivic homotopy theory}, 173-216, NATO Sci. Ser. II, Math. Phys. Chem., {\bf 131}, Kluwer Acad. Publ., Dordrecht, 2004.
\bibitem[TV08]{TV08} B. To\"en, G. Vezzosi, {\em Homotopical Algebraic Geometry II: geometric stacks and applications}, Mem. Amer. Math. Soc. {\bf 193 (902)}, 2008.

\bibitem[Ver02]{Ver} A. Verbovetsky, {\em Remarks on two approaches to the horizontal cohomology: compatibility complex and the Koszul-Tate resolution}, Acta Appl. Math. {\bf 72 (1)}, 2002, 123-131.

\bibitem[Vin01]{Vin} A. Vinogradov, {\em Cohomological Analysis of Partial Differential Equations and Secondary Calculus}, Trans. Math. Mono. {\bf 204}, American Mathematical Society, 2001.

\bibitem[Wed14]{Wed} T. Wedhorn, {\em Manifolds, sheaves, and cohomology}, University of Paderborn.

\bibitem[Wei93]{Wei93} C. A. Weibel, \emph{An introduction to homological algebra}, Cambridge studies in advanced mathematics, {\bf 38}, Cambridge University Press, {\tt ISBN 0-521-55987-1}.

\end{thebibliography}
\end{document}